\documentclass[10pt]{article}
\usepackage{amssymb, mathrsfs, amsthm} 
\usepackage[centertags]{amsmath}
\usepackage[utf8]{inputenc}
\usepackage[T1]{fontenc}
\usepackage{color}
\usepackage[english]{babel}
\usepackage[margin=1.25in]{geometry}
\usepackage{bbm}
\usepackage{color}
\usepackage{amssymb, mathrsfs, amsthm} 
\usepackage[centertags]{amsmath}
\usepackage[utf8]{inputenc}
\usepackage[T1]{fontenc}
\usepackage{color}
\usepackage[english]{babel}
\usepackage{bbm}
\usepackage{color}

 \newtheorem{thm}{Theorem}[section]

 \newtheorem{prop}[thm]{Proposition}
 \theoremstyle{definition}
 \newtheorem{defn}[thm]{Definition}
 \theoremstyle{remark}
 \newtheorem{rem}[thm]{Remark}
 \newtheorem{ass}[thm]{Assumption}
 
 \numberwithin{equation}{section}

\newcommand\de{\mathrm{d}}
\newcommand\real{\mathbb{R}}
\newcommand\DD{\mathbb{D}}

\newcommand\LL{\mathcal{L}}

\newcommand\eps{\epsilon}
\newcommand\ind{\mathbbm{1}}

\newcommand\DDD{M^2}
\newcommand\Sp{\mathcal{S}^p}
\newcommand\prob{\mathbb{P}}

\newcommand\media{\mathbb{E}}

\newcommand\KK{\mathbb{K}}

\newcommand\SDDEs{stochastic delay differential equations }

\newcommand\fb{forward-backward }
\newcommand\BDG{Burkholder-Davis-Gundy } 
\newcommand\cadlag{c\`{a}dl\`{a}g } 
\newcommand\comp{\tilde{N}}
\newcommand\sko{\hat{\de}}
\newcommand\norma{\|}

\usepackage{comment}
\usepackage[colorinlistoftodos]{todonotes}

\newcommand{\E}[1]{\mathbb{E}\left  [  #1 \right ] }
\newcommand{\R}{\mathbb{R}}

\newcommand{\RR}{\mathbb{R}}

\newcommand{\Ind}[1]{\mathbbm{1}_{\left [#1\right ]}}
\newcommand{\IndN}[1]{\mathbbm{1}_{#1}}

\newcommand{\XTE}{X^{\tau,\eta,x}}
\newcommand{\XEX}{X^{\eta,x}}

\newcommand{\EXU}{\eta_1,x_1}
\newcommand{\EXD}{\eta_2,x_2}
\newcommand{\EXZ}{\eta,x}
\newcommand{\EXUN}{\eta_1,x_1}
\newcommand{\EXDN}{\eta_2,x_2}
\newcommand{\DM}{\mathbb{D}^{1,2}}
\newcommand{\NT}[1]{\tilde{N}(#1)}

\begin{document}

%

\title{A nonlinear Kolmogorov equation for stochastic functional delay differential equations with jumps}
\author{Francesco Cordoni$^{a}$, Luca Di Persio$^{b}$,
\and Immacolata Oliva$^{c}$\bigskip\\{\small $^{a}$ Department of Mathematics, University of Trento,}\\{\small Via Sommarive, 14, Trento, 38123, Italy}\\{\small $^{b}$ Department of Computer Science, University of Verona,}\\{\small Strada le Grazie, 15, Verona, 37134, Italy}\\{\small $^{c}$ Dept. of Economics, University of Verona,}\\{\small Vicolo Campofiore, 2, Verona, 37129, Italy}}
\date{}
\maketitle


\begin{abstract}

We consider a stochastic functional delay differential equation, namely an equation whose evolution depends on its past history as well as on its present state, driven by a pure diffusive component plus a pure jump Poisson compensated measure. We lift the problem in  the infinite dimensional space of square integrable Lebesgue functions in order to show that its solution is an $L^2-$valued Markov process whose uniqueness can be shown under standard assumptions of locally Lipschitzianity and linear growth for the coefficients. Coupling the aforementioned equation with a standard backward differential equation, and deriving some ad hoc results concerning the Malliavin derivative for systems with memory, we are able to derive a non-linear 
Feynman-Kac representation theorem under mild assumptions of differentiability.
\end{abstract}

\renewcommand{\thefootnote}{\fnsymbol{footnote}}
\footnotetext{{\scriptsize E-mail addresses: francesco.cordoni@unitn.it
(Francesco Cordoni), luca.dipersio@univr.it (Luca Di Persio),
(Immacolata Oliva) immacolata.oliva@univr.it}}

\textbf{AMS Classification subjects:} 34K50; 60H07; 60H10; 60H30; 60G51 \medskip

\textbf{Keywords or phrases: }\SDDEs, quadratic variation, L\'evy processes, Feynman--Kac formula, mild solution.

\section{Introduction}\label{SEC:Intro}

During recent years, an increasing attention has been paid to stochastic equations whose evolution depends 
not only on the present state, but also on the past history. In particular, it has been shown that memory 
effects cannot be neglected when dealing with many natural phenomena. As examples, let us mention 
the coupled \emph{atmosphere-ocean models}, see, e.g., \cite{nino}, and their applications in describing 
climate changes in the environmental sciences setting, or the effect of time delay considering population dynamics, when suitable growth models are considered, see, e.g., \cite{growth}. Nevertheless, assumptions that will be made throughout the work are mainly taken into account having in mind concrete financial applications. For instance, in \cite{Pla,Swi3} the authors pointed out how delay arises in commodity markets and energy markets, when it is necessary to take into account the impact of production and transportation, whereas in \cite{Arr,Cha} the authors provide applications to option pricing in markets with memory. Similarly, delay naturally arises when dealing with financial instruments as, e.g., \textit{Asian options} or \textit{lookback options}, as studied in, e.g. \cite{CDP} and references therein. 

For the mathematical foundations of the theory of stochastic functional delay differential equations (SFDDEs) we refer to \cite{Moh}, as well as to \cite{Moh1} to many motivating examples concerning the treatment of equations with delay. In particular the monograph \cite{Moh} represents an early and deep treatment of SFDDE's, where several results concerning existence and uniqueness of solutions to SFDDE's as well as regularity results are derived. The theory of delay equations has seen a renewed attention recently, in particular in \cite{CF1,CF2} an ad hoc stochastic calculus, known as \textit{functional It\^{o}'s calculus}, has been derived, based on a suitable It\^{o}'s formula for delay equations. Also, in past few years several different works have appeared deriving fundamental results on delay equations based on semigroup theory and infinite dimensional analysis, see, e.g. \cite{Fla,Fla2}, or based on the calculus via regularization, see, e.g. \cite{Cos,FMT}. Eventually, in \cite{Fla,Cos}, it has been shown that SFDDE's, \textit{path-dependent calculus} and delay equations via semigroup theory, are in fact closely related.

Having in mind possible financial applications, the aim of the present work is to extend some results concerning the \textit{non-linear Feynman-Kac formula} for a forward-backward system with delay, where the driving noise is a non Gaussian L\'{e}vy process, using the theory of SFDDE's first introduced in \cite{Moh}. 
It is worth to mention that, particularly during last decades, asset price dynamics and, more generally,  financial instruments processes, have been widely characterized by trajectories showing sudden changes and ample jumps. It follows that the  classical Black and Scholes picture has to be refined by allowing to consider random components constituted by both  diffusive and jump components.

We thus consider  the following $\RR-$valued SFDDE with jumps
\begin{equation}\label{initialproblem}
\begin{split}
d X(t)  = & \mu(t,X(t+\cdot),X(t)) dt + \sigma(t,X(t+\cdot),X(t)) dW(t) + \\
+&\int_{\RR_0} \gamma(t,X(t+\cdot),X(t),z) \tilde{N}(dt,dz)\, ,
\end{split}
\end{equation}
where $W(t)$ is a standard Brownian motion, $\tilde{N}(dt,dz)$ is a compensated Poisson random measure with associated L\'{e}vy measure $\nu$. Also the notation $X(t+ \cdot)$ means that the {\it coefficients} $\mu, \sigma$ and $\gamma$, at time $t$, depend not only on the 
present state of the process $X$ but also on its past values.
Exploiting the concept of \textit{segment} of a process $X$, see, e.g., 
\cite{Moh,Moh1}, we will lift the finite dimensional $\real$-valued process solution to \eqref{initialproblem}  
to an infinite dimensional stochastic process with values in a suitable path-space. More precisely, in what follows we will denote by $r >0$ the maximum delay taken into account and 
$T< \infty$ a fixed finite time horizon. Thus, for an $\RR-$valued stochastic process $X,$ we indicate 
with $X(t)$ the value in $\RR$ at time $t \in [0,T]$ and with $X_t$ the corresponding segment, i.e. 
the trajectory in the  time interval $[t-r,t]$, that is $X_t(\cdot): [-r,0] \to \R$ 
is such that $X_t(\theta):= X(t+\theta)$ for all $\theta \in [-r,0]$. 

Then equation \eqref{initialproblem} can be rewritten as s
\begin{align}\label{delay_eq}
\begin{cases}
\de X(t) = &\mu(t, X_t,X(t)) \de t + \sigma(t, X_t,X(t)) \de W(t) +\\
&+ \int_{\real_0} \gamma(t,X_t,X(t),z) \comp(\de t, \de z) \\
  (X_0,X(0)) & = (\eta(\theta),x)
\end{cases}		
	\;,
\end{align}
for all $t \,\in\, [0,T],$ $\theta \, \in \, [-r,0]$, $x \in \RR$ and $\eta$ a suitable $\RR-$valued function on $[-r,0]$.
\begin{rem}
In what follows we will only consider the $1-$dimensional case, the case of a $\RR^d-$valued stochastic process, perturbed by a general $\RR^m-$ dimensional Wiener process and a $\RR^n-$dimensional Poisson random measure, with $d>1$, $m>1$ and $n>1$, can be easily obtained from the present one.
\end{rem}
In order to take into account the delay component, we study the equation \eqref{delay_eq} in the \textit{Delfour-Mitter space} defined as follows 
$\DDD := L^2\left ([-r,0];\RR\right ) \times \RR$, endowed with the scalar product
\[
\left \langle (X_t,X(t)),(Y_t,Y(t)) \right \rangle_{M_2} = \left \langle X_t,Y_t \right \rangle_{L^2} 
+ X(t) \cdot Y(t)\, ,
\]
and norm
\begin{align}\label{norm}
\|(X_t,X(t))\|_{M_2}^2 =\|X_t\|_{L^2}^2 + |X(t)|^2\, , \quad (X_t,,X(t)) \in M_2\, ,
\end{align}
where $\cdot$, resp. $|\cdot|$, stands for the scalar product in $\RR$, resp. the absolute value, and $\langle \cdot, \cdot \rangle_{L^2}$, resp. $\|\cdot\|_{L^2}$, is the scalar product, resp. norm, in $L^2([-r,0];\RR)=:L^2$. Note that the space $M^2$ is a separable Hilbert space, see, e.g., \cite{bcddr}. The \textit{Delfour-Mitter space} can be generalized to be a separable Banach space if we consider $p \in (1,\infty),$ equipped with the appropriate norm. In this work we will consider the case $p=2$.

Alternatively, we could have considered the space of \emph{\cadlag functions}, i.e. right--continuous functions with finite left limit, on the interval $[-r,0]$,  $\mathbf{D}:= D\left ([-r,0];\RR\right )$ called \textit{Skorokhod space}; in particular $\mathbf{D}$ is a non separable Banach space if endowed with the sup norm $\| \cdot \|_{\mathbf{D}} = \sup_{t \in [-r,0]} |\cdot |$. We also have that $\mathbf{D} \subset M^2$ with the injection being continuous, see, e.g., \cite{bcddr}. Nevertheless, choosing $M^2$ as state space we cannot deal with the case of discrete delays, see, e.g., \cite[pag. 3]{bcddr}, or \cite{Moh}.
 
The choice of considering the  Hilbert space $M^2$ instead of the \textit{Skorokhod space} $\mathbf{D}$ has two main motivations. First, the separability of the Hilbert space $M^2$ allows us to prove a fundamental property for the SFDDE under investigation, that is we will show that although exhibiting delay, the SFDDE \eqref{delay_eq} is a $M^2-$Markov process. The same, primary due to the fact that the \textit{Skorokhod space} $\mathbf{D}$ is not a separable Banach space, does not hold if one considers $\mathbf{D}$ as state space. One can nevertheless avoid this problem considering weaker topologies on $\mathbf{D}$, such as the so-called \textit{Skorokhod topology}, under which $\mathbf{D}$ can be shown to be separable, see, e.g \cite{Rie}. For the sake of simplicity we will address here the simpler case of an $M^2-$valued process, leaving the more technical case of $\mathbf{D}-$process to future investigations.

Second reason we are choosing here the Hilbert space $M^2$ is the extensive use we will do of \textit{Malliavin calculus}. In fact Malliavin calculus provides a powerful tool to study general regularity properties of a process or, as in the present case, to obtain representation theorem under mild regularity assumptions for the process. Nevertheless its generalization to the infinite dimensional setting, mostly when the driving noise is a general L\'{e}vy process, is rather technical and the theory, even if promising results have been obtained, see \cite{App2} and references therein, is still not completely developed. For these reasons, in the present work, we will employ an approach similar to the one used in \cite{DI} for backward stochastic differential equations with time-delayed generator and in \cite{FMT} for SFDDE with a Brownian noise. We will in fact exploit the fact that the original equation \eqref{delay_eq} has value in a finite dimensional space, so that one can use standard results in Malliavin calculus. This will imply that, exactly as in \cite{FMT}, we will not use a purely infinite dimensional formulation for our problem, such as for instance the one first formulated in \cite{ACM} and subsequently used in \cite{FT1}. In fact the $M^2-$setting will be mainly used to prove existence and uniqueness of a solution and most important, as mentioned above, we are able to prove that the SFDDE \eqref{delay_eq} is a $M^2-$Markov process.

We have already mentioned that, despite the fact that the process \eqref{delay_eq} exhibits memory effects, lifting the problem to consider a $\DDD-$value solution leads to obtain a solution which is a \emph{Markov process.} 
Taking in mind latter result and in order to derive the  \textit{Kolomogorov equation} associated to equation \eqref{delay_eq}, we will consider, following \cite{FMT,FT,FT1}, a classical $\RR-$valued backward stochastic differential equation (BSDE), coupled with the forward equation equation \eqref{delay_eq}, which evolves according to 

\begin{align}\label{Bdelay_tx}
\begin{cases}
		\de Y(t) &= \psi\left (t, X_{t},X(t), Y(t), Z(t),\int_{\RR_0} U(t,z) \delta(z) \nu(dz) \right ) \de t \\ 
		& + Z(t) \de W(t) + \int_{\real_0}U(t,z) \comp(\de t, \de z) \\
		Y(T) & = \phi(X_T,X(T)) 
\end{cases} \;,
\end{align}

where $\psi$ and $\phi$ are  given suitable functions to be specified later on. We recall that a solution to equation \eqref{Bdelay_tx} is a triplet $\left( Y, Z,U\right ),$ where $Y$ is the \emph{state process}, while $Z$ and $U$ are the \emph{control processes}.

It is well known that, when the delay is not involved, there exists a Feynman-Kac representation 
theorem that connects the solution of the coupled forward-backward system \eqref{delay_eq} and 
\eqref{Bdelay_tx}, to a deterministic semi-linear partial integro-differential equation, see, e.g., 
\cite[Chapter. 4]{delong} or \cite{Bar} for further details.
When the delay is taken into consideration, previous result has been recently proved in the Brownian case in \cite{FMT,FT1}. 
In the present paper we extend latter result taking into consideration a non Gaussian L\'{e}vy noise. 
In particular, exploiting notations already introduced, we will consider the following coupled forward-backward stochastic differential equation (FBSDE) with delay, for $t \in [\tau,T] \subset [0,T],$

\begin{align}\label{syst}
\begin{cases}
    \de X^{\tau,\eta,x}(t) & = \mu(t, X^{\tau,\eta,x}_t,X^{\tau,\eta,x}(t)) \de t + \sigma(t, X^{\tau,\eta,x}_t,X^{\tau,\eta,x}(t)) \de W(t) \\
		 & \quad  + \int_{\real_0} \gamma(t,X_t^{\tau,\eta,x},X^{\tau,\eta,x}(t),z) \comp(\de t, \de z) \\
		\smallskip
  	 (X_\tau^{\tau,\eta,x},  X^{\tau,\eta,x}(\tau))  &= (\eta,x )\in \DDD \\
		\de Y^{\tau,\eta,x}(t) & = \psi\left (t, X^{\tau,\eta,x}_{t},X^{\tau,\eta,x}(t), Y^{\tau,\eta,x}(t) 
		,Z^{\tau,\eta,x}(t), \tilde{U}^{\tau,\eta,x}(t)\right ) \de t\\
		& \quad + Z^{\tau,\eta,x}(t) \de W(t) +  \int_{\real_0}U^{\tau,\eta,x}(t,z) \comp(\de t, \de z) \\
		Y^{\tau,\eta,x}(T)  &= \phi(X_T^{\tau,\eta,x},X^{\tau,\eta,x}(T)) 
\end{cases}
\, ,
\end{align}

where we have denoted for short by
\[
\tilde{U}^{\tau,\eta,x}(t) := \int_{\RR_0} U^{\tau,\eta,x}(t,z)\delta(z) \nu(dz)\, .
\]
Moreover we have denoted by $X^{\tau,\eta,x}$  the value of the process with starting time $\tau \in [0,T]$ and initial value $(\eta,x) \in \DDD$. In what follows we will often omit the dependence on the initial value point $(\eta,x)$ and we assume that  the process starts at time $\tau=0$, i.e. $X_t^{0,\eta,x} =: X_t$. Also, in order to simplify notation, most of the results will be proved for $\tau=0$, the extension 
to the general case of $\tau \not =0$  being straightforward.

We are going to connect the solution to the FBSDE \eqref{syst} to the solution of the following partial integro-differential Hilbert--space valued equation

\begin{equation}\label{EQN:MildKolmo}
\begin{cases}
\frac{\partial}{\partial t}u(t,\eta,x) + \mathcal{L}_t u(t,\eta,x) 
= \psi\left (t,\eta,x, u(t,\eta,x), \partial_x u(t,\eta,x) \sigma(t,\eta,x), \mathcal{J} u(t,\eta,x)\right  )\, \\
u(T,\eta,x) = \phi(\eta,x)\, ,\quad t \in [0,T]\,\,\,,\,\,\, (\eta,x)\in \DDD\, .
\end{cases}
\end{equation}

where $\mathcal{L}_t$ is the \emph{infinitesimal generator} of the forward $\DDD-$valued process in equation \eqref{FBsyst_tx}, $\partial_x$ is the derivative with respect to the present state $X(t)$ and $ \mathcal{J} $ is the operator
\[
\mathcal{J} u(t,\eta,x) := \int_{\RR_0} [u(t, \eta,x+  \gamma(t,\eta,x,z))-u(t,\eta,x)] \delta(z)\nu (dz)\, .
\]

In particular, we will consider a mild notion of solution to equation \eqref{EQN:MildKolmo}, so that we say that a function $u : [0,T] \times \DDD \to \RR$ is a \textit{mild solution} to equation \eqref{EQN:MildKolmo} 
if there exist $C >0$ and $m \geq 0$, such that, for any $t \in [0,T]$ and any $(\eta_1,x_1)$, $(\eta_2,x_2) \in \DDD$, 
$u$ satisfies

\begin{equation}
\begin{split}
&|u(t,\eta_1,x_1) - u(t,\eta_2,x_2)|\leq C|(\eta_1,x_1)-(\eta_2,x_2)|_2(1+ |(\eta_1,x_1)|_2 + |(\eta_2,x_2)|_2 )^m\, ,\\
&|u(t,0,0)| \leq C \, ,
\end{split}
\end{equation}

and the following equality holds true
\begin{equation}\label{def_mild2}
u(t,\eta,x) = P_{t,T} \phi (\eta,x) + \int_t^T P_{t,s}[\psi(\cdot, u(s,\cdot), \partial_x u(s,\cdot) 
\sigma(s,\cdot), \mathcal{J} u(s, \cdot )](\eta,x) \de s \;,
\end{equation}
for all $t\, \in \, [0,T],$ and $(\eta,x) \, \in \, \DDD$, $P_{t,s}$ being the \emph{Markov semigroup} 
related to the equation \eqref{delay_eq}. In particular we would like to stress that we require the solution $u$ to equation \eqref{EQN:KFTTh} to be locally Lipschitz continuous with respect to the second variable with at most polynomial growth, so that the derivative appearing in the right--hand--side of equation \eqref{def_mild2} is to be defined in a mild sense, to better specified later on.

We thus define  
\begin{align*}
\begin{cases}
 Y^{\tau,\eta,x}(t)&:=u(t,X^{\tau,\eta,x}_t,X^{\tau,\eta,x}(t))\\
Z^{\tau,\eta,x}(t)&:=\partial_x u(t,X^{\tau,\eta,x}_t,X^{\tau,\eta,x}(t)) \,\,\sigma(t,X^{\tau,\eta,x}_t,X^{\tau,\eta,x}(t)) \\
 U^{\tau,\eta,x}(t,z)&:=u\left (t, X^{\tau,\eta,x}_t,X^{\tau,\eta,x}(t)+\gamma(t,X^{\tau,\eta,x}_t,X^{\tau,\eta,x}(t),z) \right ) \\
 & - u\left (t,X^{\tau,\eta,x}_t,X^{\tau,\eta,x}(t)\right )
\end{cases}
\end{align*}
then the triplet $\left (Y^{\tau,\eta,x},Z^{\tau,\eta,x},U^{\tau,\eta,x}\right )$ is the unique 
solution to the backward equation \eqref{Bdelay_tx}, where $\partial_x$ is the derivative with respect to
the $\RR-$valued present state $X(s)$ of $(X_s,X(s))$, $u$ being  the \textit{mild solution} to the \textit{Kolmogorov equation} 

\[
\begin{cases}
\frac{\partial}{\partial t}u(t,\eta,x) + \mathcal{L}_t u(t,\eta,x) 
= \psi\left (t,\eta,x, u(t,\eta,x), \partial_x u(t,\eta,x) \sigma(t,\eta,x), \mathcal{J} u(t,\eta,x)\right  )\, \\
u(T,\eta,x) = \phi(\eta,x)\, ,\quad t \in [0,T]\,\,\,,\,\,\, (\eta,x)\in \DDD \, .
\end{cases}
\, .
\]

As regard the notion of mild solution for the \textit{Kolmogorov equation} \eqref{EQN:MildKolmo}, we have to mention that different notions can be chosen. Our choice is due mainly to the fact that, since we do not require for differentiability assumptions, it seems to be the most suitable for financial applications. 
As an example, in option pricing one usually have that the terminal payoff of a given claim is Lipschitz continuous, without being differentiable.
Moreover, mild differentiability assumptions and the use of delayed coefficients, allow the above notion to be particularly suited to price exotic options, as in the case of \textit{Asian options}, see \cite{CDP}. Furthermore, the notion of mild solution we have chosen well emphasize the intrinsic stochastic nature of the problem, also providing an immediate connection to BSDE theory, hence allowing to treat general semilinear PIDE.
We refer to \cite{delong} for a comprehensive treatment of BSDE's with general L\'{e}vy noise, see also \cite{CDPBSDE} and references therein for a more financially oriented study of the topic.
 
We would also like to recall that different notions of mild solution for partial integro-differential already exist in literature, mostly considering \textit{Volterra-type} equations, allowing also to exhibit delays, we refer the interested reader to \cite{Mild1,Mild2,Mild3}. Also, in a setting similar to the present one, a notion of mild solution for SPDE's driven by $\alpha-$stable noise can be found in \cite{KLSu,KLSu2}, where the authors study mild solutions of semilinear parabolic equations in an infinite dimensional Hilbert space, in order to obtain the associated Hamilton-Jacobi-Bellman equation with applications to the stochastic optimal control problems.

Last but not least, rather recently a further notion of mild solution to delay equations has appeared in literature. This is an ad hoc generalization of the standard notion of \textit{viscosity solution}. In particular first in \cite{Vis1}, and then in \cite{Vis2,Vis3}, a new notion of viscosity solution to PDE with delays, called \textit{path-dependent PDE}, has been formulated, based on the newly developed \textit{functional It\^{o} calculus} mentioned above. Latter  notion has been also exploited to treat \textit{path-dependent} PDE with delayed generator, see, e.g., \cite{CDMZ}, or \cite{CDPCVA} for an application to mathematical finance.

The paper is organized as follows: in Section \ref{sfdd} we introduce necessary notations and formalize the tools necessary to treat delay equations in the Hilbert space $M^2$. In particular Section \ref{sfdd} is devoted to the characterization of fundamental results on SFDDE, such as existence and uniqueness, as well as the Markov property of the forward process. Thus subsection \ref{mall_jumps_d} is devoted to results concerning Malliavin calculus for delay equations which will be needed in order to prove the main representation theorem. In Section \ref{jointqv} we prove the main result based on Malliavin calculus, which is related to the study of the  joint quadratic variation of the forward equation and a suitable function; in Section \ref{k_eq} we give the \textit{non-linear Feynman-Kac theorem} that is later used to derive a deterministic representation to the FBSDE. Finally in Section \ref{SEC:OC} we give an application of obtained result to optimal control.

\section{Forward stochastic functional differential equation with delay} \label{sfdd}

In this Section we introduce the notation used throughout the paper, also presenting basic definitions and
main results related to the mathematical techniques involved in our approach. Some results are already established in literature, such as existence and uniqueness of solutions, whereas others  
are proved here for the first time.

Let us consider a probability space $(\Omega, \mathcal{F}, \left (\mathcal{F}_t\right )_{t\in [0,T]}, \prob)$, where 
$\left (\mathcal{F}_t\right )_{t\in [0,T]}$ is the \emph{natural} filtration jointly generated by the random variables 
$W(s)$ and $N(ds,dz),$ for all $z\in \mathbb{R}\setminus \{0\}=:\R_0$ and for all $s\in [0,T],$ augmented 
by all $\mathbb{P}$-null sets,  $W$ being a $1$-dimensional Brownian motion, while $N$ is 
a $1$-dimensional Poisson random measure, independent from $W$,  with associated L\'{e}vy measure $\nu(dz)$, satisfying
\begin{equation}\label{EQN:ConNu1}
\int_{\RR_0} \min \{1, z^2\} \nu(dz) < \infty \, ;
\end{equation}
also we define the compensated random measure $\tilde{N}(dt,dz):=N(dt,dz)-\nu(dz)dt$.

We will further assume in what follows that the L\'{e}vy measure $\nu$ satisfies
\begin{equation}\label{EQN:ConNu2}
\int_{\RR_0} |z|^2 \nu(dz) < \infty \, .
\end{equation}

We underline that condition \eqref{EQN:ConNu1} is a standard assumption in the definition of a L\'{e}vy measure $\nu$, whereas assumption \eqref{EQN:ConNu2} implies that the process has a finite second moment, which is a natural assumption if one has in mind financial applications.

In the following, we fix a {\it delay} $r >0$ and we will use the notation $X(t)$ to denote the present state 
at time $t$ of the real valued process $X$, whereas we use $X_t$ to denote the {\it segment} of the 
path described by $X$ during the time interval $[t-r,t]$ with values in a suitable infinite dimensional 
path space. In particular, we refer to the couple
\[
\left (\left (X(t+\theta)\right )_{\theta \in [-r,0]}, X(t)\right ) =: \left (X_t, X(t)\right )\, .
\] 
From now on, we define $\DDD := L^2 \times \RR := L^2([-r,0]; \real)\times \RR$, endowed with the scalar product
\[
\left \langle (X_t,X(t)),(Y_t,Y(t)) \right \rangle_{M_2} = \left \langle X_t,Y_t \right \rangle_{L^2} + X(t) \cdot Y(t)\, ,
\] 
and norm
\begin{equation}\label{normD}
\|\left (X_t, X(t)\right )\|_{\DDD}^2 = \|X_t\|_{L^2}^2+|X(t)|^2\, , 
\end{equation} 
namely the \textit{Delfour-Mitter space}, which is a separable Hilbert space, see, e.g., \cite{Moh} 
and reference therein for details. 

Furthermore, for any $p \in [2,\infty)$, we denote by $\Sp(t) := \Sp([0,t];\DDD)$ and we say 
that a $\DDD-$valued stochastic process $\left (X_s,X(s)\right )_{s \in [0,t]}$ belongs to $\Sp(t)$ if 
\[
\|X \|^p_{\Sp(t)} := \media \left[\sup_{s \,\in\, [0,t]} \|(X_{s},X(s))\|_{\DDD}^p \right]  < \infty \;.
\]
We denote for short $\Sp := \Sp(T)$. For the sake of simplicity, the following notation is used throughout 
the paper: $| \cdot |_2$ denotes the norm in $\DDD$ and $| \cdot |$ the absolute value in $\RR$.

\begin{rem}
Let us stress that we will consider here a $\RR-$valued stochastic process $X$, nevertheless any result that follows can be easily generalized to the case of an $\RR^d-$ valued stochastic process. In particular we would have considered the \textit{Delfour-Mitter space} $M^2([-r,0];\RR^d) := L^2([-r,0];\RR^d) \times \RR^d $, see, e.g. \cite{bcddr}.
\end{rem}

As briefly said in Section \ref{SEC:Intro}, the main goal of this work is to study a stochastic functional 
delay differential equation (SFDDE) of the form

\begin{equation}\label{EQN:DelayGenerale}
\begin{cases}
    \de X(t) = \mu(t, X_{t}, X(t) ) \de t + \sigma(t, X_{t}, X(t) ) \de W(t) + \int_{\real_0} \gamma(t,X_{t}, X(t) ,z) 
		\comp(\de t, \de z) \\
  	(X_0,X(0)) = (\eta,x) \in \DDD 
\end{cases}		
	\;,
\end{equation}

for all $t \,\in\, [0,T]$. We will assume the functionals $\mu, \sigma$ and $\gamma$ to fulfil the following assumptions.
 

\begin{ass}\label{ass1}
\begin{itemize}
\item[(A1)] the coefficients 
\begin{equation*}
\mu: [0,T] \times \DDD  \to \RR\, ,\quad  \sigma: [0,T] \times \DDD \to \RR\, , \quad 
\gamma :[0,T] \times \DDD \times \RR_0 \to \RR
\end{equation*}
are continuous. 

\item[(A2)] There exists $K > 0$ such that for all $t \,\in\, [0,T]$ and for all $(\eta_1,x_1)$, 
$(\eta_2,x_2) \,\in\, \DDD,$
\begin{align*}
|\mu(t,\EXUN) & - \mu(t,\EXDN)|^2 + |\sigma(t,\EXUN) - \sigma(t,\EXDN)|^2 \\ 
& + \int_{\RR_0} |\gamma(t,\EXUN,z) -\gamma(t,\EXDN,z)|^2 \nu(dz) \\ 
 & \leq K|(\EXU)-(\EXD)|_2^2 (1 + |(\EXU)|_{2}^2+|(\EXD)|_2^2)\, .
\end{align*}
\end{itemize} 
\end{ass}
%

Throughout the paper, we will look for \emph{strong solution} to equation \eqref{EQN:DelayGenerale} in the following sense.

\begin{defn}
We say that $X:= (X_t,X(t))_{t \in [0,T]}$ is a \textit{strong solution} to equation \eqref{EQN:DelayGenerale} 
if for any $t \in [0,T]$ $X$ is indistinguishably unique and $\left (\mathcal{F}_t \right )_{t\in [0,T]}$-adapted and it holds $\mathbb{P}-$a.s.
\[
\begin{split}
X(t) &= x  + \int_0^t \mu(s,X_s,X(s)) ds + \int_0^t \sigma(s,X_s,X(s))dW(s) \\ 
 & \qquad + \int_0^t \int_{\R_0} \gamma(s,X_s,X(s),z) \tilde{N}(ds,dz) \, ,\\
 X_0 &= \eta\, .
\end{split}
\]
\end{defn}
\noindent
In what follows we will denote by $\left(\XTE_t,\XTE(t)\right) $ the $\DDD-$value of the process at 
time $t \in [\tau,T]$, with initial value $(\EXZ) \in \DDD$ at initial time $\tau \in [0,T]$. However, for the sake of brevity, in most of the results, we will avoid to state the dependence on the initial value $(\tau,\eta, x)$ writing for short $\left (X_t,X(t)\right )$ instead of $\left(\XTE_t,\XTE(t)\right) $.

Now we provide an existence and uniqueness result for equation \eqref{EQN:DelayGenerale}. 

\begin{thm}\label{eu_f}
Suppose that $\mu$, $ \sigma$ and $\gamma$ satisfy conditions $(A1)-(A2)$ in Assumptions \ref{ass1}. Then, 
for all $t \in [0,T]$ and $(\eta,x) \in  \DDD,$ there exists a unique strong solution to the SFDDE \eqref{delay_eq} 
in $\Sp$ and there exists $C_1 :=C_1(K,L,T,p)$ such that
\begin{equation}\label{stima1}
\norma \XEX \norma^{p}_{\Sp} \leq C_1(1 + |(\EXZ)|_{2}^p)\, .
\end{equation}
Moreover, the map $(\EXZ) \mapsto \XEX$ is Lipschitz continuous from $\DDD$ to $\Sp$ and it exists 
$C_2:=C_2(K,L,T)$ such that 
\begin{equation}\label{EQN:LipStima}
\norma X^{\EXU} - X^{\EXD} \norma^{p}_{\Sp} \leq C_2 |(\EXU) - (\EXD)|_{2}^p\,.
\end{equation}
\end{thm}

\begin{proof}
Existence and  uniqueness of the solution to equation \eqref{EQN:DelayGenerale}, as well as the estimate in equation \eqref{stima1}, are proved in \cite[Th. 2.14]{bcddr}.

As regards equation \eqref{EQN:LipStima}, exploiting the \BDG inequality, see, e.g. \cite[Section 4.4.]{App}, we have that, for any $t \in [0,T]$, denoting for short by $C$ several positive constants,
\begin{align*}
| & X^{\EXU} -X^{\EXD}|_{\Sp}^{p} =\\
&= \mathbb{E} \sup_{t \in [0,T]} |(X_t^{\EXU},X^{\EXU}(t))-(X_t^{\EXD},X^{\EXD}(t))|_2^p \leq\\
&\leq C |(\EXU) -(\EXD)|_2^p \\
 &+ C \left[ \int_0^t |\mu(s,X_s^{\EXU},X^{\EXU}(s))-\mu(s,X_s^{\EXD},X^{\EXD}(s))|^p ds \right. \\
 &\left. + \left (\int_0^t |\sigma(s,X_s^{\EXU},X^{\EXU}(s))-\sigma(s,X_s^{\EXD},X^{\EXD}(s))|^2 ds\right )^{\frac{p}{2}} \right. \\
 & \left. + \int_0^t \int_{\RR_0}|\gamma(s,X_s^{\EXU},X^{\EXU}(s),z)-\gamma(s,X_s^{\EXD},X^{\EXD}(s),z)|^p 
\nu(dz) ds \right] \, ,
\end{align*}
so that from the Lipschitz continuity in assumption \ref{ass1} $(A2)$, it follows
\[
\begin{split}
& \mathbb{E} \sup_{t \in [0,T]} |(X_t^{\EXU},X^{\EXU}(t))-(X_t^{\EXD},X^{\EXD}(t))|_2^p \leq\\
&\leq C |(\eta_1,x_1) - (\eta_2,x_2)|^p_2 +\\
&+ \int_0^T \sup_{s \in [0,q]}|(X_s^{\EXU},X^{\EXU}(s))-(X_s^{\EXD},X^{\EXD}(s))|_2^p ds \, ,
\end{split}
\]
and the claim follows from Grownall's inequality.
\end{proof}

\begin{rem}
We want to stress that a result analogous to Thm. \ref{eu_f} can be obtained by replacing the Delfour-Mitter space 
$\DDD$ with the space $\mathbf{D}$ of \cadlag functions, with the corresponding sup norm 
$\| \cdot \|_{\mathbf{D}} = \sup_{t \,\in\, [-r,0]} |\cdot|,$ see e.g. \cite{bcddr,Rie}. 
\end{rem}

One of the major results, when one is to lift the delay equation into an infinite dimensional setting exploiting the notion of \textit{segment}, is that one is able to recover the Markov property of the driving equation, see, e.g \cite[Theorem II.1]{Moh1}. Similarly also equation \eqref{EQN:DelayGenerale} results to be 
an $\DDD-$valued Markov process.

\begin{prop}\label{PRO:Markov}
Let $X =\left ((X_t,X(t))\right )_{t \in [0,T]}$ be the strong solution to equation \eqref{EQN:DelayGenerale}, then the process $X$ is a Markov process in the sense that
\[
\prob((X_t,X(t)) \,\in\, B | \mathcal{F}_s) = \prob((X_t,X(t)) \,\in\, B | (X_s,X(s))=(\eta,x)) \, ,\quad \mathbb{P}-a.s. \;,
\]
for all $0 \leq s \leq t \leq T$ and for all Borel sets $B \in \mathcal{B}(\DDD)$.
\end{prop}
\begin{proof}
See, e.g. \cite[Th. 3.9]{bcddr}, or also, see, e.g., \cite[Prop. 3.3]{Rie} or \cite[Sec.9.6]{Zab}.
\end{proof}

Having shown in Proposition \ref{PRO:Markov} that $X$ is a $\DDD-$valued Markov process, we can therefore introduce the transition semigroup $P_{\tau,t},$ acting on the space of Borel bounded function on $\DDD,$ denoted by $B_b(\DDD)$, namely, we define
\begin{equation}\label{semigroup}
P_{\tau,t} : B_b(\DDD) \,\rightarrow\, B_b(\DDD)\, ,\quad  P_{\tau,t}[\varphi](x) := \media[\varphi(X^{\tau,\eta,x}_{t})]\, , \quad  \quad \varphi \in B_b(\DDD)\, .
\end{equation}

Concerning the infinitesimal generator $\mathcal{L}_t$ of equation \eqref{EQN:DelayGenerale}, following \cite{FMT,FT,FT1}, we will not enter in further details concerning its explicit representation or the characterization of its domain, since this goes beyond the aim of the present work and it is not necessary in order to prove the main results. Nevertheless let us mentioned that its form can be derived from a direct application of It\^{o}'s formula, see, e.g. \cite[Th. 3.6]{bcddr}.

%
%
%


\subsection{Malliavin calculus for jump processes with delay} \label{mall_jumps_d}

In this subsection we recall some definitions and main results concerning Malliavin operator and 
Skorokhod integral for jump processes. We will give fundamental definition in order to fix the 
notation and to recall the most effective results, we refer to \cite{mall,Pet} for further references 
and proofs of some results, or to \cite{OksRos,DI} for application of Mallavin calculus to delay equations.

In particular we stress that very few results concerning Malliavin calculus for jump processes in infinite dimension exist, where also the most simple case of jumps processes having values in an infinite dimensional Hilbert space is difficult to treat, we refer the interested reader to \cite{App2}. In order to avoid problems coming with the Hilbert space setting we will, in the present section, exploit the same ideas used in \cite{DI}. Using the fact that the original SDE has finite dimensional realizations. This will allow us to exploit standard results in Malliavin calculus for jumps processes with values in $\RR^d$. Also, in order to be able to do so, as in \cite{DI}, we must work with delay of integral type, which motivates the choice of the Hilbert space $M^2$.

In order to keep the present paper as much as self contained as possible, we will first recall definitions and fundamental results for Malliavin calculus for jumps processes mainly taken from \cite{mall}. Eventually we state the main result of the present subsection, that is, as done in \cite{DI} exploiting the finite dimensional nature of the SFDDE, we prove a Malliavin differentiability result for SFDDE. Also, for the sake of brevity, we will state the results just for the jump component and 
we refer to \cite{FT1,Moh1} for the diffusive part. 

Let us denote by $ I_n(f) $ the $n$-fold \emph{iterated stochastic integral} w.r.t. the random measure $\tilde{N},$ as

\begin{equation}\label{iter_int} 
I_n(f_n) := \int_{([0,T]\times \real_0)^n} f((t_1,z_1), \ldots, (t_n,z_n))\comp(\de t_1, \de z_1) \ldots 
\NT{\de t_n,\de z_n} \in L^2(\Omega)\;,
\end{equation}

where 
\[
f \in L^2\left (([0,T]\times \real_0)^n\right ) = L^2\left (([0,T]\times \real_0)^n),
\otimes \nu(dz)dt\right )\, ,
\] 
is a deterministic function. 

Thus, every random variable $F \in L^2(\Omega)$ can be represented as an infinite sum of iterated integrals 
of the form \eqref{iter_int}. This representation is known as \textit{chaos expansion}, see, e.g.\cite[Def. 12.1]{mall} or \cite[Th. 1]{Pet}.

\begin{thm}
The stochastic Sobolev space $\DD^{1,2}$ consists of $\mathcal{F}-$measurable random variable $F \in L^2(\Omega)$ such that, for $(f_n)_{n \geq 0},$ with $f_n \in L^2\left (([0,T]\times \real_0)^n\right ),$ it holds
\begin{equation} \label{chaos}
F= \sum_{n=0}^\infty I_n(f_n)\, ,
\end{equation}
with the following norm
\[
\| F \|^2 = \sum_{n=0}^\infty n n!\| I_n(f_n) \|_{ L^2\left (([0,T]\times \real_0)^n\right )}^2\, .
\]
\end{thm}

Given the \textit{chaos expansion} in equation \eqref{chaos}, we can introduce the \textit{Malliavin derivative} $D_{t,z}$ and its domain $\DM$, see, e.g. \cite[Def. 12.2]{mall}.

\begin{defn}
Let us consider a random variable $F \,\in\, \DD^{1,2},$ the \textit{Malliavin derivative} is the operator $D : \DD^{1,2} \subset L^2(\Omega) \,\to 
L^2(\Omega\times[0,T]\times\real_0)$ defined as
\begin{align}\label{m_der_t}
D_{t,z}F & = \sum_{n=1}^{\infty} n I_{n-1}(f_n(\cdot,t,z)), \; F \,\in\, \DD^{1,2}, \, z \neq 0 \, .
\end{align}
\end{defn}
Since the operator $D$ is closable, see, e.g., \cite[Thm. 3.3 and Thm 12.6]{mall}, we denote by $\DD^{1,2}$ the domain 
of its closure.

The following result represents a \emph{chain rule} for Malliavin derivative. 
\begin{thm}\label{THM:ChainMall2}
Let $F \,\in\, \DD^{1,2}$ and let $\phi$ be a real continuous function on $\real.$ Suppose $\phi(F) \,\in\, L^2(\Omega)$ 
and $\phi(F + D_{t,z}F) \,\in\, L^2(\Omega\times[0,T]\times\real_0).$ Then, $\phi \,\in\, \DD^{1,2}$ and 
\begin{equation}\label{chain}
D_{t,z}\phi(F) = \phi(F + D_{t,z}F) - \phi(F) \;.
\end{equation}
\end{thm}
\begin{proof}
See, e.g. \cite[Thm. 12.8]{mall}.
\end{proof}

Once the Malliavin derivative has been defined, we are able to introduce its adjoint operator, 
the \textit{Skorokhod integral}, in particular next definition is taken from \cite[Def. 11.1]{mall}, see, also \cite[Sec. 3]{Pet} for details. 

\begin{defn}\label{DEF:M1}
Let $\delta : L^2(\Omega\times [0,T] \times \RR_0) \to L^2(\Omega)$ be the adjoint operator of the derivative $D.$ 
The set of processes $h \in  L^2(\Omega\times [0,T] \times \RR_0) $ such that
\[
\left | \mathbb{E} \int_0^T \int_{\RR_0} D_{s,z} F\, h_t(z)\, \nu(dz) ds \right | \leq C \|F\|\, ,
\]
for all $F \in \DM$, forms the domain of $\delta$, denoted by $dom\, \delta$.

For every $h \in dom \, \delta$ we can define the \emph{Skorokhod integral} as
\[
\delta(h) := \int_0^T \int_{\RR_0} h_t(z) \tilde{N}(\hat{d}t,dz)\, ,
\]
for any $ F \in \DM$.
\end{defn}

\begin{defn}\label{DEF:M2}
We denote by $\mathbb{L}^{1,2}$ the space of $\mathcal{F}-$adapted processes $h :\Omega \times [0,T] \times 
\RR_0 \to \RR$ such that $h_t \in \mathbb{D}^{1,2}$ and 
\begin{align*}
\mathbb{E} \int_0^T \int_{\RR_0} |h_t(z)| \nu(dz) dt & < \infty\, \\
\mathbb{E} \int_{\left ([0,T] \times \RR_0\right )^2}|D_{t,z}h_s(\zeta)| \nu(d\zeta) ds \nu(d z) dt & < \infty\, .
\end{align*}
\end{defn}

From Definitions \ref{DEF:M1}--\ref{DEF:M2} above, we have that $\mathbb{L}^{1,2} \subset dom \, \delta$. 
If $h \in \mathbb{L}^{1,2}$ and $D_{t,z} h \in dom \, \delta$,  then $\delta(h) \in \mathbb{D}^{1,2}$ and 
\begin{equation}\label{EQN:DeltSk}
D_{t,z} \delta(h) = h(z) + \delta(D_{t,z} h)\, ,
\end{equation}
see, e.g. \cite{Sol2}. Notice also that $\mathbb{L}^{1,2} \simeq L^2 ([0,T];\mathbb{D}^{1,2})$. 

\begin{prop}\label{PRO:MallInt}
Let $h_t$ be a predictable square integrable process. Then, if $h \in \DM,$ we have, for a.e. $(s,z)\in [0,t] \times \RR_0$,
\begin{align*}
D_{s,z} \int_0^t h_\tau d\tau &  =  \int_s^t D_{\tau,z} h_\tau d\tau \, ,\\
D_{s,z} \int_0^t h_\tau dW(\tau) &  =  \int_s^t D_{\tau,z} h_\tau dW(\tau)  \, ,\\
D_{s,z} \int_0^t\int_{\RR_0} h_\tau \tilde{N}(d\tau,dz) & = h_s + \int_s^t \int_{\RR_0} 
D_{\tau,\zeta} h_\tau \tilde{N}(d\tau,d \zeta)  \,.
\end{align*}
\end{prop}
\begin{proof}
ee, e.g. \cite[Prop. 6]{Pet}.
\end{proof}

Next result is the chain rule for SFDDE, that is the generalization of Theorem \ref{THM:ChainMall2} to the case of delay equations, that will be needed in the proof of the main result of the present Section as well as in subsequent sections.

\begin{thm}\label{THM:ChainMallF}
Let $F$ and $\psi \,\in\, \DD^{1,2}$, let also $\phi$ be a real valued continuous function on $M^2$ Suppose $\phi(\psi,F) \,\in\, L^2(\Omega)$  and $\phi(\psi + D_{t,z} \psi, F + D_{t,z}F) \,\in\, L^2(\Omega\times[-r,T]\times\real_0).$ Then, $\phi \,\in\, \DD^{1,2}$ and it holds
\begin{equation} \label{eq:chain_rule}
D_{t,z} \phi(\psi,F) = \phi(\psi + D_{t,z} \psi, F + D_{t,z} F) - \phi(\psi,F) \;.
\end{equation}
\end{thm}

\begin{proof}
Following \cite[Proposition 6.2]{Fen}, let us define a partition of $[-r, 0]$,
\[
\Pi_k : -r \leq s_1 < \ldots < s_k \leq 0\, ,
\]
with 
\[
\|\Pi_k\| := \max_{2 \leq i \leq k}(s_i - s_{i-1}) \,\rightarrow\, 0, \quad \mbox{ as } \, k \,\rightarrow\, \infty\, .
\]

Let $I_k: \real^k \,\rightarrow\, L^2([-r,0],\real)$ be the continuous linear embedding associated to 
the partition $\Pi_k$ as 
\[
I_k(x_1, \ldots, x_k)(t) := \sum_{i=1}^k x_i I_{(s_{i-1}, s_i]}(t) \;,
\] 
and set $\underline{s}^k$ the tuple $(s_1, \ldots, s_k)$. Let us also define 
\[
Q_{\underline{s}^k}(\psi) := \left(\frac{1}{s_1 - s_0} \int_{s_0}^{s_1} \psi(t) \de t, \ldots, 
\frac{1}{s_k - s_{k-1}} \int_{s_{k-1}}^{s_k} \psi(t) \de t \right)\, ,
\]
the $L^2$ projection for $\psi \,\in\, L^2([-r,0],\real)$. Finally, we define a linear map $T^k:L^2([-r,0],\real) \to L^2([-r,0],\real)$ as
\[
T^k \, : \, \psi \mapsto \psi^k := T^k \psi := I_k \circ Q_{\underline{s}^k}(\psi)\, ;
\]
in particular it holds that $T^k \psi \to \psi$ in $L^2([-r,0],\real)$ as $k \to \infty$, see, e.g. \cite[Lemma 5.1]{Fen}.

We thus define the function $\phi^k : \RR^k \times \RR \to \RR$, so that, from the classical chain rule Theorem \ref{THM:ChainMall2} applied to $\phi^k$ we have
\[
\begin{split}
D_{t,z} \phi(\psi^k,F)&= D_{t,z} \phi^k\left ( Q_{\underline{s}^k}(\psi),F\right )=\\
&=\phi^k( Q_{\underline{s}^k}\left (\psi + D_{t,z} \psi\right ), F + D_{t,z} F) - \phi^k( Q_{\underline{s}^k}\left (\psi\right ),F) =\\
&=\phi(I_k \circ Q_{\underline{s}^k}\left (\psi+ D_{t,z} \psi\right ), F + D_{t,z} F) - \phi(I_k \circ Q_{\underline{s}^k}\left (\psi\right ),F)\, .
\end{split}
\]

Then the claim follows taking the limit as $k \to \infty$ together with \cite[Lemma 5.1]{Fen}, the continuity of $\phi$ and the Dominated Convergence Theorem.
\end{proof}

We are finally able to prove next theorem, which is the main result of the current subsection concerning Malliavin differentiability of the SFDDE \eqref{EQN:DelayGenerale}.

\begin{thm}\label{THM:MallDiff}
Let us suppose that Assumptions \ref{ass1} (A1)-(A2) hold and $X=\left (X(t)\right )_{t \in [-r,T]}$ is the solution to equation \eqref{EQN:DelayGenerale}. Then, 
$X \,\in\, L^2\left ([-r,T]; \mathbb{D}^{1,2} \right )$ and, for every $s \in [0,T]$ and $z \in \RR_0,$ the stochastic process $\{D_{s,z} X(t) \, : \, t \in [s,T]\}$ satisfies
\begin{equation}\label{EQN:A}
\E{\int_0^T \int_{\RR_0} \sup_{t \in [s,T]}|D_{s,z} X(t)|^2 \nu(dz)ds } < \infty\, \, .
\end{equation}

In particular, for any $t \in [0,T]$, $X(t) \in \mathbb{D}^{1,2}$ and it holds

\begin{equation}\label{EQN:MallDiffSFDDE1}
\begin{cases}
D_{s,z} X(t) &= \gamma(s,X_s,X(s),z) + \\
&+\int_s^t \left (\mu\left (X_u + D_{s,z}X_u,X(u) + D_{s,z}X(u)\right ) - \mu(X_u,X(u))\right ) du+ \\ 
& +\int_s^t \left (\sigma\left (X_u + D_{s,z}X_u,X(u) + D_{s,z}X(u)\right ) - \sigma(X_u,X(u)) \right )dW(u) +\\
&+ \int_s^t \int_{\RR_0} \left ( \gamma\left (X_u + D_{s,z}X_u,X(u) + D_{s,z}X(u)\right ) - \gamma(X_u,X(u))\right ) \tilde{N}( du,d \zeta)\, ,\\
D_{s,z} X(t) &= 0\, ,\quad t \in [-r,s)\, , \\
\end{cases}
\, ,
\end{equation}

Moreover, for any $z \in \RR_0$, there exists a measurable version of the two-parameter process
\[
D_{s,z} X_t = \left \{D_{s,z} X_t(\theta) \, : \, s \in [0,T] \, , \theta \in [-r,0]  \right  \}\, .
\]
\end{thm}
\begin{proof}
We will use a standard Picard's approximation scheme, see, e.g. \cite[Th. 17.2]{mall}. Let $X^0(t) = x$ and $X^0_t = \eta$, then set, for $n > 0$,
\begin{align*}
X^{n+1}(t) & = x + \int_0^t \mu(s,X^n_s,X^n(s)) ds +\int_0^t \sigma(s,X^n_s,X^n(s)) d W(s) \\ 
 & + \int_0^t \int_{\RR_0} \gamma(s,X^n_s,X^n(s),z) \tilde{N}(ds,dz)\, ,\\
X^{n+1}_0 &= \eta\, ,
\end{align*}
where we use the notation $X^n_s := \left (X^n(s+\theta)\right )_{\theta \in [-r,0]}$.

We are going to prove by induction over $n$ that $X^n(t) \in \mathbb{D}^{1,2}$ for any $t \in [0,T]$, 
$D_{s,z} X(t)$ is a predictable process and that
\[
\xi_{n+1}(t) \leq C_1 + C_2 \int_{-r}^t \xi_n(s) ds\, ,
\]
where $C_1, \,C_2$ are some suitable constants and
\[
\xi_{n}(s) := \sup_{0 \leq s \leq t}\mathbb{E} \int_{\RR_0} \sup_{s \leq \tau \leq t} 
|D_{s,z} X^n(\tau)|^2 \nu(dz)< \infty\, .
\]

For $n= 0$ the above claim is trivially satisfied. Let us thus assume that the previous assumptions 
hold for $n$, we have to show that they hold also for $n+1$. 
Indeed we have that $\int_0^t \mu(s,X^n_s,X^n(s)) ds$, $\int_0^t \sigma(s,X^n_s,X^n(s)) dW(s)$ 
and $\int_0^t \gamma(s,X^n_s,X^n(s),z) \tilde{N}(ds,dz) \in \mathbb{D}^{1,2}$, 
and proposition \ref{PRO:MallInt} guarantees that 
\begin{align*}
D_{s,z} \int_0^t \mu(\tau,X^n_\tau,X^n(\tau)) d\tau &  =  \int_s^t D_{\tau,z} \mu(\tau,X^n_\tau,X^n(\tau)) d\tau \\ 
D_{s,z} \int_0^t \sigma(\tau,X^n_\tau,X^n(\tau)) dW(\tau) & =  \int_s^t D_{\tau,z} \sigma(\tau,X^n_\tau,X^n(\tau)) dW(\tau)
\end{align*}
and
\begin{align*}
D_{s,z} & \int_0^t \gamma(\tau,X^n_\tau,X^n(\tau),z) \tilde{N}(d\tau,dz) = \gamma(s,X^n_s,X^n(s),z) \\ 
& + \int_s^t \int_{\RR_0} D_{\tau,\zeta} \gamma(\tau,X^\tau_s,X^n(\tau),\zeta) \tilde{N}(d\tau,d \zeta) 
\end{align*}
for $s \leq t.$
Consequently, for any $t \in [0,T],$ $X^{n+1}(t) \in \mathbb{D}^{1,2}$ and
\begin{align}
\nonumber 
D_{s,z} X^{n+1}(t) & = \gamma(s,X^n_s,X^n(s),z) +\int_s^t D_{\tau,z}\mu(\tau,X^n_\tau,X^n(\tau)) d\tau \\ \nonumber 
& +\int_s^t D_{\tau,z} \sigma (\tau,X^n_\tau,X^n(\tau)) dW(\tau) \\ \label{EQN:eQN1}
& + \int_s^t \int_{\RR_0} D_{\tau,z} \gamma (\tau,X^\tau_\tau,X^n(\tau),\zeta) \tilde{N}(d\tau,d \zeta)\, ,
\end{align}
and the representation in equation \eqref{EQN:MallDiffSFDDE1} immediately follows from the chain rule Th. \ref{THM:ChainMallF}.

By squaring both sides of equation \eqref{EQN:eQN1}, we have
\begin{align}
\nonumber
\left |D_{s,z} X^{n+1}(t)\right |^2 & \leq  4\left | \gamma(s,X^n_s,X^n(s),z)\right 
|^2 +\left |\int_s^t \mu_{\tau,z}(\tau,X^n_\tau,X^n(\tau)) d\tau \right |^2 \\ \nonumber
& +4 \left |\int_s^t \sigma_{\tau,z}(\tau,X^n_\tau,X^n(\tau)) dW(\tau)\right |^2 \\ \label{EQN:MallPJum}
& + 4 \left | \int_s^t \int_{\RR_0} \gamma_{\tau,z}(\tau,X^\tau_s,X^n(\tau),\zeta) 
\tilde{N}(d\tau,d \zeta)\right |^2\, .
\end{align}

By exploiting Doob maximal inequality, stochastic Fubini's theorem and It\^{o} isometry, 
we get

\begin{equation}\label{EQN:MallPJum2}
\begin{split}
\mathbb{E}\int_{\RR_0} & \sup_{s \leq \tau \leq t}\left |D_{s,z} X^{n+1}(t)\right |^2 \nu(dz) 
\leq  C\left [\mathbb{E}\int_{\RR_0} \left | \gamma(s,X^n_s,X^n(s),z)\right |^2 \nu(dz)\right . \\
&+\mathbb{E}\left |\int_s^t \mu_{\tau,z}(\tau,X^n_\tau,X^n(\tau)) d\tau \right |^2 
+\mathbb{E}\left |\int_s^t \sigma_{\tau,z}(\tau,X^n_\tau,X^n(\tau)) dW(\tau)\right |^2\\
&+\left .\mathbb{E}\left | \int_s^t \int_{\RR_0} \gamma_{\tau,z}(\tau,X^\tau_s,X^n(\tau),\zeta) 
\tilde{N}(d\tau,d \zeta)\right |^2\right ]\ \\
&\leq  C\left [\mathbb{E}\int_{\RR_0} \left | \gamma(s,X^n_s,X^n(s),z)\right |^2 \nu(dz) \right. \\
& + \left. \mathbb{E}\int_s^t \left | \mu_{\tau,z}(\tau,X^n_\tau,X^n(\tau))\right |^2 d\tau 
+ \mathbb{E}\int_s^t \left | \sigma_{\tau,z}(\tau,X^n_\tau,X^n(\tau)) \right |^2 d \tau \right. \\ 
& + \left .\mathbb{E} \int_s^t \int_{\RR_0} \left | \gamma_{\tau,z}(\tau,X^\tau_s,X^n(\tau),\zeta) 
\right |^2 \nu(dz) d\tau\right ]\, ,
\end{split}
\end{equation}

where we denote for short by $C>0$ a suitable constant.

Exploiting Assumptions \ref{ass1} together with Theorem \ref{THM:ChainMall2}, we get 

\begin{equation}\label{EQN:StimFinaMl}
\begin{split}
&\mathbb{E}\int_{\RR_0} \sup_{s \leq \tau \leq t} |D_{s,z} X^{n+1}(\tau)|^2 \nu(dz) \leq \\
&\leq C_1 \int_{s}^t \mathbb{E} \int_{\RR_0} |\left (D_{s,z} X^n_{\tau},D_{s,z} X^n(\tau)\right ) |^2_2 \nu(dz) d\tau + C_2\left (1+ \mathbb{E}  |\left ( X_{\tau}^{n}, X^{n}(\tau)\right ) |^2_2\right ) \\ 
& \leq C_1 \left (\mathbb{E} \int_{s}^t \int_{-r}^0 \int_{\RR_0} \left |D_{s,z} X^n(\tau+\theta)\right |^2 \nu(dz) d \theta d\tau  
+ \mathbb{E} \int_{s}^t \int_{\RR_0} \left |D_{s,z} X^n(\tau)\right |^2 d \nu(dz) \tau\right ) +\\
&+C_3 (1+ \lambda) \\
& \leq C_1 \left (\mathbb{E} \int_{-r}^0 \int_{s}^{t+\theta} \int_{\RR_0} \left |D_{s,z} X^n(p)\right |^2 \nu(dz) dp d\theta 
+ \mathbb{E} \int_{s}^t \int_{\RR_0} \left |D_{s,z} X^n(\tau)\right |^2 \nu(dz) d\tau\right ) +\\
&+C_3 (1+ \lambda) \\
&\leq C_4 \mathbb{E} \int_{s}^t \int_{\RR_0} \left |D_{s,z} X^n(\tau)\right |^2 \nu(dz) d\tau +C_3 (1+ \lambda) \, ,
\end{split}
\end{equation}

where $C_1, \, C_2$, $C_3$ and $C_4$ denote some suitable constants and $\lambda$ is such that
\[
\lambda = \sup _{n} \mathbb{E} \sup_{-r \leq s \leq T} |X^n(s)|^2_2 < \infty\, .
\]

Also, we obtain 
\[
X^{n+1} = \left (X^{n+1}(t)\right )_{t \in [-r,T]} \in L^2(\Omega \times [-r,T])\, ,
\]
and for any $t$, $X^{n+1}(t) \in \DM$, so that $X^{n+1} \in L^2(\Omega \times [-r,T];\DM)$ 
and, for $p \leq s, \; D_{s,z} X^{n+1}(p) = 0$. 

It follows that, for any $z \in \RR_0$, it exists a measurable version of the two-parameter process
\[
D_{s,z} X^{n+1}_t = \left \{D_{s,z} X^{n+1}_t(\theta) \, : \, s \in [0,T] \, , \theta \in [-r,0] \right  \}\, ,
\]
such that $D_{s,z} X^{n+1}_t \in L^2(\Omega \times [0,T] \times [-r,0])$, see, e.g. \cite[Sec. 4]{Moh1}.

Therefore, the inductive hypothesis is fulfilled by $X^{n+1}$ and 
\[
\mathbb{E} \sup_{s \leq T} |X^n(s) - X(s)|^2 \to 0 \, \quad \mbox{ as } \,\, n \to \infty\,.
\]

Finally, thanks to a discrete version of Gronwall's lemma, see, e.g. \cite[Lemma 4.1]{Bel} or \cite[Th. 17.2]{mall}, 
and applying equation \eqref{EQN:StimFinaMl}, we have 
\[
\sup_{n \geq 0}\mathbb{E} \int_{-r}^T |D_{s,z} X^n(\tau)|^2 d \tau < \infty\, ,
\]
so that $X(t) \in \mathbb{D}^{1,2}$.

By repeating the same reasoning as before, we have 
$$X = \left (X(t)\right )_{t \in [-r,T]} \,\in\, L^2(\Omega \times [-r,T]), \; X(t) \,\in\, \DM \;,$$ 
for any $t,$ so that $X \in L^2(\Omega \times [-r,T];\DM)$. 
The proof is complete by observing that, for any $z \in \RR_0$, there exists a measurable version of 
the two-parameter process
\[
D_{s,z} X_t = \left \{D_{s,z} X_t(\theta) \, : \, s \in [0,T] \, , \theta \in [-r,0] \right  \}\, ,
\]
such that $D_{s,z} X_t \in L^2(\Omega \times [0,T] \times [-r,0])$.
\end{proof}

\section{Joint quadratic variation}\label{jointqv}

In order to prove the main result of this work, which consists in giving an explicit \textit{Feynman-Kac} 
representation formula for a coupled forward-backward system with delay, we need first to prove a joint quadratic variation result. 
The main advantage of such an approach is to overcome difficulties that may arise in dealing with the It\^{o} 
formula in infinite dimension, since, in general, the process $X_t$ fails to be a semi-martingale, so we cannot 
rely on standard It\^{o} calculus. Furthermore, with the present approach, we are able to relax hypothesis concerning the differentiability 
of the coefficients.

Following \cite{FMT,FT1}, we introduce a \textit{generalized covariation} process. The definition of joint generalized quadratic variation we consider in the present paper has been first introduced in \cite{Rus2}, see also \cite{Rus,Rus3}, with the only difference that they consider the limit to hold uniformly on compacts sets in probability. We have chosen here, following \cite{FMT,FT1}, to consider the limit in probability because the limiting procedure is easier with a stronger notion of convergence, such as the convergence in probability. Also, it is shown in \cite[Prop. 1.1]{Rus2} that the standard definition of joint quadratic variation, see, e.g. \cite[Section 4.4.3]{App}, coincides with the quadratic variation defined below. 

\begin{defn}\label{DEF:JointQV}
Given a couple of $\real$-valued stochastic processes $(X(t),Y(t)),$ $t \geq 0,$ we define their \emph{joint quadratic 
variation} on $[0,T]$, to be
\[
\left\langle X(t),Y(t) \right\rangle_{[0,T']} := \mathbb{P}-\lim_{\eps \downarrow 0} C^{\eps}_{[0,T']}(X(t),Y(t))\;, 
\]
where $\mathbb{P}-\lim$ denotes the limit to be taken in probability and
\begin{equation}\label{joint}
C^{\eps}_{[0,T']}(X(t),Y(t)) := \frac{1}{\eps}\int_0^{T'} (X(t+\epsilon) - X(t))(Y(t+\epsilon) - Y(t)) \de t, \; \eps > 0\;,
\end{equation}
with $0\leq  T' + \epsilon < T$.
\end{defn}

Before stating our main result we are to better introduce a mild notion of derivative we will use throughout the paper. In what follows we will consider a function $u : [0,T] \times \DDD \to \RR,$ such that there exist $C >0$ and $m \geq 0$, such that, for any $t \in [0,T]$ and any $(\eta_1,x_1)$, $(\eta_2,x_2) \in \DDD$, 
$u$ satisfies

\begin{equation}
\begin{split}
&|u(t,\eta_1,x_1) - u(t,\eta_2,x_2)|\leq C|(\eta_1,x_1)-(\eta_2,x_2)|_2(1+ |(\eta_1,x_1)|_2 + |(\eta_2,x_2)|_2 )^m\, ,\\
&|u(t,0,0)| \leq C \, .
\end{split}
\end{equation}

that is we require the function $u$ to be Lipschitz continuous without requiring any further regularity concerning differentiability. Nevertheless, in what follows, we will use the notation of $\partial_x^\sigma$. In particular following \cite{FT1} we will introduce a mild notion of derivative, called \textit{generalized directional gradient} $\partial_x^\sigma u$. When $u$ is sufficiently regular, it can be shown that the \textit{generalized directional gradient}, in the direction $\sigma(t,\eta,x)$, of a function $u$, coincides with $\partial_x u(t,\eta,x) \sigma(t,\eta,x)$. 

The definition, as well as the characterization of several properties, for the \textit{generalized directional gradient} has been provided in \cite{FT1}. We will only state here the definition of \textit{generalized directional gradient}, whereas we refer to \cite{FT1} to a complete treatment of the topic.

In particular it has been shown in \cite{FT1} that the following holds 
\begin{equation}\label{EQN:GenGra}
\left \langle u(\cdot, X_{\cdot}, X(\cdot),W(\cdot) \right \rangle_{\tau,t} = \int_\tau^t \zeta(s,X_s,X(s)) ds \;,
\end{equation}
where $\left \langle \, \cdot\,,\, \cdot \right \rangle_{\tau,t}$ denotes the \textit{joint quadratic variation} defined above and $\zeta:[0,T] \times \DDD \to \RR$ is a suitable measurable map, see also \cite{FMT,FT,FT1} for details. 
Under suitable hypothesis of regularity, in \cite{FT1} the authors show that 
\begin{equation}\label{EQN:GeGPS}
\left \langle u(\cdot, X_{\cdot}, X(\cdot),W(\cdot) \right \rangle_{\tau,t} = \int_\tau^t \partial_x u(t,X_s,X(s)) \sigma(t,X_s,X(s)) ds \, ,\quad \mathbb{P}-a.s.\, ,
\end{equation}
where we denote by $\partial_x$ the derivative w.r.t. the present state. Hence, 
equation \eqref{EQN:GenGra} can be considered  as the definition of the \textit{generalized directional 
gradient} of the function $u$ along the direction $\sigma$. We say that the map $\zeta:[0,T] 
\times \DDD \to \RR$ belongs to the directional gradient of $u$, or equivalently that $\zeta \in 
\partial_x^\sigma u$, if equation \eqref{EQN:GenGra} holds. 
Therefore, we use for short the notation $\partial_x^\sigma u$ to represent an element of the \textit{generalized 
directional gradient.} Since this topic lies outside our goals, having been deeply studied in a more general 
setting in \cite{FT1}, we skip every technicality and invite the interested reader to \cite{FT1}. 

The following result represents the core of this paper. 

\begin{thm}\label{thm3.1}
Let us assume that $u : [0,T] \times \DDD \to \RR$ is locally Lipschitz w.r.t. the second variable and with 
at most polynomial growth, namely, there exist $C >0$ and $m \geq 0$, such that, for any $t \in [0,T]$ and 
any $(\eta_1,x_1)$, $(\eta_2,x_2) \in \DDD$, $u$ satisfies

\begin{equation}\label{EQN:KFTTh}
\begin{split}
&|u(t,\eta_1,x_1) - u(t,\eta_2,x_2)|\leq C|(\eta_1,x_1)-(\eta_2,x_2)|_2(1+ |(\eta_1,x_1)|_2 + |(\eta_2,x_2)|_2 )^m\, ,\\
&|u(t,0,0)| \leq C \, .
\end{split}
\end{equation}

Then, for every $(\EXZ) \,\in\, \DDD$ and $0 \leq \tau \leq  T' \leq T,$ the process
\[
\{u(t,X_t^{(\tau,\EXZ)},\XTE(t)), \, t \,\in\, [\tau,T']\}
\]
admits a joint quadratic variation on the interval $[\tau,T']$ with 
$$J(t):= \int_0^t \int_{\RR_0} z \comp(ds,dz),$$ 
given by

\begin{equation}\label{EQN:JV2}
\begin{split} 
&\left\langle u(\cdot, X_{\cdot}^{\tau,\EXZ}, \XTE (\cdot)), J(\cdot) \right\rangle_{[\tau,T']} =\\ =\int_\tau^{T'} &\int_{\real_0} z \left[u(s, X_s^{\tau,\EXZ},\XTE (s)+ \gamma(s,\XTE_s,\XTE(s),z) - u(s, X_s^{\tau,\EXZ},\XTE(s)) \right] N(\de s, \de z)\, .
\end{split}
\end{equation}

\end{thm}
\begin{rem}
An analogous of \cite[Prop. 4.4]{FT1} is valid in the present case, that is the following representation holds
\[
\left\langle u(\cdot, X_{\cdot}^{\tau,\EXZ}, X^{\tau,\EXZ}(\cdot)), W(\cdot) \right\rangle_{[\tau,T']} 
= \int_\tau^{T'} \partial_x^\sigma u(s, X_s^{\tau,\EXZ},\XTE(s) )ds\, ,
\]
where $\partial_x^\sigma u$ is the \textit{generalized directional gradient}. The claim follows from \cite{FT1} 
by observing that the Poisson random measure does not affect the result and the proof follows exactly 
the same steps as in \cite{FT1}.
\end{rem}
\begin{proof}
Without loss of generality, we prove the result for $\tau = 0$, as the case of a general 
initial time $ \tau \not=0 $ can be proved using the same techniques. 
Fix $(\eta,x) \,\in\, \DDD$ and a time horizon $T' \,\in\, [0,T]$ and denote for brevity 
$X^{0,\eta,x}$ by $X$. In what follows we will denote with $\tilde{N}(\hat{dt},dz)$ the Skorokhod integral.

In order to shorten the notation set
\[
v_t := (u(t+\epsilon,X_{t+\eps},X(t+\epsilon)) - u(t,X_t,X(t))\Ind{0,T'}(t)\, ,
\]
and 
\[
A^\eps := \{ (t,s)\in [0,T']\times [0,T']\, : \, 0 \leq t \leq T'\, , \, \, t \leq s \leq t+\eps\}\, .
\]

From equation \eqref{EQN:KFTTh} and theorem \ref{THM:MallDiff}, we have $v_t \in \mathbb{L}^{1,2},$ 
so that, for any $ t $, $ v_t \in \DD^{1,2} $ and then $ v_t \IndN{A^\eps}(t,\cdot) \in L^2(\Omega \times [0,T]).$ 
Furthermore, equation \eqref{EQN:DeltSk} implies that $v_t$ is Skorokhod integrable and from \cite[Th. 12.11]{mall} we have
\begin{align}
\nonumber
\int_0^{T'} \int_{\RR_0} z v_t \IndN{A^\eps}(t,s)\comp(\hat{dt},dz) & = v_t \int_0^{T'} \int_{\RR_0} 
z \IndN{A^\eps}(t,s)\comp(\hat{dt},dz) \\ \label{EQN:1}
& -  \int_0^{T'} \int_{\RR_0} z D_{s,z} v_t \IndN{A^\eps}(t,s)N(ds,dz) =: z_t\, ,
\end{align}
which holds since $z \in L^2(\Omega \times [0,T])$. Also, equation \eqref{EQN:1} implies, for a.a. $ t \in [0,T'],$
\begin{equation}\label{EQN:2}
\begin{split}
& u(t+\eps,X_{t+\eps},X(t+\epsilon)) -u(t,X_t,X(t)) (J_{t+\eps}-J_t) \\
&= u(t+\eps,X_{t+\eps},X(t+\epsilon))-u(t,X_t,X(t)) \int_t^{t+\eps} \int_{\RR_0} z \comp(ds,dz) \\
& = \int_t^{t+\eps} \int_{\RR_0} z D_{s,z}\left (u(t+\eps,X_{t+\eps},X(t+\epsilon)) -u(t,X_t,X(t)) \right ) N(ds,dz) \\
&+ \int_t^{t+\eps} \int_{\RR_0} z\left (u(t+\eps,X_{t+\eps},X(t+\epsilon)) -u(t,X_t,X(t)) \right ) \comp(ds,dz)\, .
\end{split}
\end{equation}

Let us integrate the right-hand side of equation \eqref{EQN:2} in $[0,T']$ w.r.t. $t.$ 
By noticing that the left-hand side equals to $\epsilon C^\epsilon$, we write the right-hand side 
as follows
\begin{equation}\label{EQN:3}
\begin{split}
& \int_0^{T'} \int_t^{t+\eps} \int_{\RR_0} z D_{s,z}\left (u(t+\eps,X_{t+\eps},X(t+\epsilon)) 
-u(t,X_t,X(t)) \right ) N(ds,dz) dt \\
&+\int_0^{T'} \int_t^{t+\eps} \int_{\RR_0} z \left (u(t+\eps,X_{t+\eps},X(t+\epsilon)) 
-u(t,X_t,X(t)) \right ) \comp(\hat{ds},dz) dt \\
&= \int_0^{T'} \int_t^{t+\eps} \int_{\RR_0} z D_{s,z}\left (u(t+\eps,X_{t+\eps},X(t+\epsilon)) 
-u(t,X_t,X(t)) \right ) N(ds,dz) dt \\
&+ \int_0^{T'+\eps}\int_{\RR_0}  \int_{(s-\eps)^+}^{s \wedge T'}  z \left (u(t+\eps,X_{t+\eps},X(t+\epsilon)) 
-u(t,X_t,X(t)) \right )dt  \comp(\hat{ds},dz) \, . 
\end{split}
\end{equation}

It remains to verify that $ \int_0^{T'} z v_t \IndN{A^\eps}(t,\cdot)dt$ appearing in equation \eqref{EQN:3} 
is Skorokhod integrable. From the definition of Skorokhod integral, by using equation \eqref{EQN:1} for 
$ G \in \DD^{1,2} $ and the \emph{duality formula}, see e.g. \cite[equation (12.14)]{mall}, we have 
\[
\begin{split}
& \E{\int_0^{T} \int_{\RR_0} \int_0^{T} z v_t \IndN{A^\eps}(t,s) dt D_{s,z} G \nu(dz)ds} \\
&= \int_0^{T} \E{ \int_0^{T} \int_{\RR_0} z v_t \IndN{A^\eps}(t,s) D_{s,z} G \nu(dz)ds}dt \\
& = \int_0^{T} \E{ G  \int_0^{T}\int_{\RR_0} z v_t \IndN{A^\eps}(t,s) D_{s,z} 	\comp(\hat{ds},dz) }dt 
= \E{G\int_0^{T} z_t dt}\, ,
\end{split}
\]

so that $ \int_0^{T'} v_t \IndN{A^\eps}(t,\cdot)dt$ is Skorokhod integrable. Hence, 
\begin{align*}
\int_0^{T} & \int_0^{T} \int_{\RR_0} z v_t \IndN{A^\eps}(t,s) dt \comp (\hat{ds},dz) \\ 
 & = \int_0^{T}z_t dt = \int_0^{T}\int_0^{T} \int_{\RR_0} z v_t \IndN{A^\eps}(t,s) \comp (\hat{ds},dz) dt \, .
\end{align*}

Exploiting again equation \eqref{EQN:1} we have
\begin{align*}
\int_0^{T}\int_0^{T'} & \int_{\RR_0} z v_t \IndN{A^\eps}(t,s) dt \comp (\hat{ds},dz) 
= \int_0^{T} z v_t (J_{t+\eps}-J_t) \Ind{t,T}(t) dt \\ 
& - \int_0^{T'} \int_0^{T}\int_{\RR_0} z D_{s,z}v_t \IndN{A^\eps}(t,s) N(ds,dz) dt\, ,
\end{align*}
and then equation \eqref{EQN:3} is proved.

On the other hand, thanks to the chain rule Theorem \ref{THM:ChainMallF} and from theorem \ref{THM:MallDiff} together with the adeptness property of the Malliavin derivative, i.e. $D_{s,z}X(t) = 0$ if $s >t$, we have that, for a.a. $ s \in [t,t+\eps] $,
\begin{align*}
D_{s,z} v_t &= D_{s,z}[u(t+\eps, X_{t+\eps},X(t+\epsilon))-u(t,X_t,X(t))] \\ 
& = D_{s,z}[u(t+\eps, X_{t+\eps},X(t+\epsilon))] \\
&= u(t+\eps, X_{t+\eps} + D_{s,z}X_{t+\eps},X(t+\epsilon) + D_{s,z} X(t+\epsilon)) \\ 
& - u(t+\eps , X_{t+\eps},X(t+\epsilon))\;. 
\end{align*}
Now, we apply equation \eqref{EQN:3} to get

\[ 
\begin{split}
& C^{\eps} =\frac{1}{\eps} \int_0^{T'}\int_t^{t + \eps} \int_{\real_0} z\left [ u\left (t+\eps, X_{t+\eps} + D_{s,z}X_{t+\eps},X(t+\epsilon) + D_{s,z} X(t+\epsilon)\right ) N(\sko s, \de z) \de t\right .\\
&-\frac{1}{\eps} \int_0^{T'}\int_t^{t + \eps} \left . u\left (t+\eps , X_{t+\eps},X(t+\epsilon)\right )\right ] N(\sko s, \de z) \de t  \\
& + \frac{1}{\eps} \int_0^{T'+\eps} \int_{\real_0}  \int_{(s-\eps)^+}^{s \wedge T'} z\left ( u(t+\eps, X_{t+\eps},X(t+\epsilon))- u(t,X_t,X(t)) \right ) dt \comp(\sko s, \de z)\, .
\end{split} 
\]

Let us consider separately the two terms

\[
\begin{split}
I_1^{\eps} &:= \frac{1}{\eps} \int_0^{T'}\int_t^{t + \eps} \int_{\real_0} z\left [ u\left (t+\eps, X_{t+\eps} + D_{s,z}X_{t+\eps},X(t+\epsilon) + D_{s,z} X(t+\epsilon)\right ) N(\sko s, \de z) \de t\right .\\
&-\frac{1}{\eps} \int_0^{T'}\int_t^{t + \eps} \left . u\left (t+\eps , X_{t+\eps},X(t+\epsilon)\right )\right ] N(\sko s, \de z) \de t \, ,\\
I_2^\epsilon &:= \frac{1}{\eps} \int_0^{T'+\eps} \int_{\real_0}  \int_{(s-\eps)^+}^{s \wedge T'} z\left ( u(t+\eps, X_{t+\eps},X(t+\epsilon))- u(t,X_t,X(t)) \right ) dt \comp(\sko s, \de z) \, .
\end{split}
\]

As regards $I_2^{\eps}$, the proof proceed  as in \cite[Prop. 4.4.]{FT1}, see also 
\cite[Th. 3.1]{FMT}. We report in what follows its main steps for the sake of completeness.
We have to show that
\[
\frac{1}{\eps} \int_0^{T'}  v_t \ind_{A^{\eps}}(t,s) \de t \to 0\, ,
\]
in $\mathbb{L}^{1,2}$, since this implies $I_2^{\eps} \to 0$ in $L^2(\Omega),$ together with the boundedness of the Skorokhod 
integral. Thus, for a general $y \in \mathbb{L}^{1,2},$ we have
\[
T^\epsilon (y)_s = \frac{1}{\eps} \int_0^{T'} ( y_{t+\epsilon}-y_t) \ind_{A^{\eps}}(t,s)dt = \frac{1}{\eps} \int_{(s-\epsilon)\vee t}^{s\wedge T} ( y_{t+\epsilon}-y_t) dt\, ,
\]
so that we have to show that $T^\epsilon (y) \to 0 $ in $\mathbb{L}^{1,2}$. 

Let us recall the isomorphism
\[
L^2\left ([0,T]; \mathbb{D}^{1,2}(\RR)\right ) \simeq \mathbb{L}^{1,2}\, .
\]

Following \cite{FT1}, we have to prove that $\|T^\epsilon\|_{\mathbb{L}^{1,2}}(\RR)$ is 
bounded uniformly w.r.t. $\epsilon$. 
In fact, we have
\[
\begin{split}
\|T^\epsilon(y)_s\|^2_{\mathbb{D}^{1,2}(\RR)}  &\leq \frac{1}{\epsilon^2}\int_0^{T'}\ind_{A^{\eps}}(t,s)dt \int_0^{T'} |y_{t+\epsilon}-y_t|_{\mathbb{D}^{1,2}(\RR)}^2\ind_{A^{\eps}}(t,s)dt \\
&\leq \int_0^{T'} |y_{t+\epsilon}-y_t|_{\mathbb{D}^{1,2}(\RR)}^2\ind_{A^{\eps}}(t,s)dt \, ,\\
\|T^\epsilon(y)_s\|^2_{\mathbb{L}^{1,2}(\RR)}  &=  \int_0^{T'}\|T^\epsilon(y)_s\|^2_{\mathbb{D}^{1,2}(\RR)} ds \\ 
& \leq \int_0^{T'}  |y_{t+\epsilon}-y_t|_{\mathbb{D}^{1,2}(\RR)}^2 \int_0^{T'} \ind_{A^{\eps}}(t,s)ds \,dt  \\
&\leq  \int_0^{T'} |y_{t+\epsilon}-y_t|_{\mathbb{D}^{1,2}(\RR)}^2 dt \leq 2 \|y\|^2_{\mathbb{L}^{1,2}(\RR)}\, ,
\end{split}
\] 
and thus the claim follows by \cite[Prop. 4.4.]{FT1}, or \cite[Th. 3.1]{FMT}. 

As regards $ I_1^{\eps},$ we have

\begin{align*}
I_1^{\eps} & = \frac{1}{\eps}\int_{0}^{T} \int_t^{t+\eps} \int_{\real_0} z \, u(t+\epsilon,X_{t+\eps} 
+ D_{s,z}X_{t+\eps},X(t+\epsilon) + D_{s,z} X(t+\epsilon)) N(\sko s, \de z) \de t \\
& - \frac{1}{\eps}\int_{0}^{T} \int_t^{t+\eps} \int_{\real_0} z \, u(t+ \epsilon,X_{t+\eps},X(t+\epsilon))]  
N(\sko s, \de z) \de t := K_1^{\eps} - K_2^{\eps} \;.
\end{align*}

Let us first prove that
\begin{equation}\label{EQN:K2}
K_2^{\eps} \,\to  \int_0^{T'} \int_{\real_0} z u(t,X_t,X(t)) N(\sko t, \de z) \;,\quad  \mathbb{P}-a.s.
\end{equation}
as $\epsilon \to 0$.

From assumption \eqref{EQN:KFTTh} on the function $u$, the right-continuity of $X$, and exploiting  the Lebesgue differentiation theorem together with  the dominated convergence theorem, it follows that 
\begin{equation}\label{EQN:Convergenza1}
\begin{split}
&\frac{1}{\eps}\int_{0}^{T'} \int_t^{t+\eps} \int_{\real_0} z  u(t+\epsilon,X_{t+\eps},X(t+\epsilon)) N(\sko s, \de z) \de t \\
 &=\int_{0}^{T'+\epsilon} \int_{\RR_0} z \frac{1}{\epsilon} \int_{s \vee \epsilon}^{(s+\epsilon) \wedge T'} u(t,X_{t},X(t)) dt N(\sko s,dz) \\ 
& \to  \int_0^{T'} \int_{\real_0} z u(s,X_s,X(s))N(\sko s, \de z) \;,
\end{split}
\end{equation}
$\mathbb{P}-$a.s., as $\epsilon \to 0$.

Let us now prove that
\begin{equation}\label{EQN:FinalConv}
K_1^{\eps} \to  \int_0^{T'} \int_{\real_0} z u(t,X_t,X(t) + \gamma(t,X_t,X(t),z)) N(\sko t, \de z) \, .
\end{equation}

Theorem \ref{THM:MallDiff} assures that
\begin{equation}\label{EQN:Mall1}
\begin{split}
D_{s,z}X(t+\eps)  =&  \gamma(s,X_s,X(s),z) + \int_s^{t+\eps} D_{s,z}[\mu(q,X_q,X(q))] \de q \\ 
& + \int_s^{t+\eps} D_{s,z}[\sigma(q,X_q,X(q))] \de W(q)  \\
& + \int_s^{t+\eps} \int_{\real_0} D_{s,z}[\gamma(q,X_q,X(q),\zeta)] \comp(\de q, \de \zeta)\, . \\ 
\end{split}
\end{equation}

Proceeding as above, we get 

\begin{equation}\label{EQN:Convergenza16}
\begin{split}
&\frac{1}{\eps}\int_{0}^{T} \int_t^{t+\eps} \int_{\real_0} z  u(t+\epsilon,X_{t+\eps}+D_{s,z}X_{t+\epsilon}) 
N(\sko s, \de z) \de t \\
& = \int_{0}^{T'+\epsilon} \int_{\RR_0} z \frac{1}{\epsilon} \int_{(s-\epsilon)^+}^{s \wedge T'} 
u \left(t+\epsilon,X_{t+\eps}+D_{s,z}X_{t+\epsilon},X(t+\epsilon) + D_{s,z} X(t+\epsilon)\right) dt N(\sko s,dz)\, .
\end{split}
\end{equation}

The continuity of $u$, together with the right-continuity of $X$, and the Lebesgue differentiation theorem provide that 

\begin{equation}\label{EQN:FinalConv3}
\begin{split}
\int_{\epsilon}^{T'+\epsilon} \int_{\RR_0} z \frac{1}{\epsilon}\int_{s \vee \epsilon}^{(s+\epsilon) \wedge T'} & u(t,X_{t}+D_{s,z}X_{t},X(t)+D_{s,z}X(t)) dt N(\sko s,dz) \\
&\to \int_0^{T'} \int_{\real_0} z u(t,X_t + D_{t,z} X_t,X(t) + D_{t,z} X(t)) N(\sko t, \de z)\, ,
\end{split}
\end{equation}

$\mathbb{P}-$a.s. as $\epsilon \to 0$.

Moreover, theorem \ref{THM:MallDiff} implies that  
\[
\begin{split}
D_{s,z}X(t + \theta) & = \gamma(s,X_s,X(s),z) + \int_s^{t+\theta} D_{s,z}[\mu(q,X_q,X(q))] \de q \\ 
& + \int_s^{t+\theta} D_{s,z}[\sigma(q,X_q,X(q))] \de W(q)  \\
& + \int_s^{t+\theta} \int_{\real_0} D_{s,z}[\gamma(q,X_q,X(q),\zeta)] \comp(\de q, \de \zeta)\, , \quad \theta \in [-r,0]\, ,\\
D_{s,z}  X(t + \theta)&= 0 \, ,\quad s > t + \theta\, ,
\end{split}
\]
and exploiting the adaptedness of the Malliavin derivative, namely
\begin{equation}\label{EQN:AdMall}
\begin{split}
D_{t,z}X_{t}(\theta) &= D_{t,z} X(t+\theta)= 0\, ,\quad \mbox{ for }\theta \in [-r,0)\, ,\\
D_{t,z} X(t) &= \gamma(t,X_t,X(t),z)\, ,
\end{split}
\end{equation}
and substituting  \eqref{EQN:AdMall} into eq. \eqref{EQN:FinalConv3}, we obtain the claim and  \eqref{EQN:FinalConv} is proved.
%
Equation \eqref{EQN:JV2} thus follows and the proof is then complete.
\end{proof}

\section{Existence of mild solutions of Kolmogorov equation}\label{k_eq}

The main goal of this section is to prove an existence and uniqueness result of a \textit{mild} solution, 
in a sense to be specified later, of a non-linear path-dependent partial integro-differential equation. 
Such a solution is connected to a forward-backward system with delay of the form 

\begin{align}\label{FBsyst_tx}
\begin{cases}
    \de X^{\tau,\eta,x}(t) & = \mu(t, X^{\tau,\eta,x}_t,X^{\tau,\eta,x}(t)) \de t + \sigma(t, X^{\tau,\eta,x}_t,X^{\tau,\eta,x}(t)) \de W(s) \\
		 & \qquad  + \int_{\real_0} \gamma(t,X_t^{\tau,\eta,x},X^{\tau,\eta,x}(t),z) \comp(\de t, \de z) \\
		\smallskip
  	(X_\tau^{\tau,\eta,x},  X^{\tau,\eta,x}(\tau))  &= (\eta,x )\in \DDD \\
		\de Y^{\tau,\eta,x}(t) & = \psi\left (t, X^{\tau,\eta,x}_{t},X^{\tau,\eta,x}(t), Y^{\tau,\eta,x}(t) ,
		Z^{\tau,\eta,x}(t),\tilde{U}^{\tau,\eta,x}(t)\right ) \de t\\
		&\qquad + Z^{\tau,\eta,x}(t) \de W(t) +  \int_{\real_0}U^{\tau,\eta,x}(t,z) \comp(\de t, \de z) \\
		Y^{\tau,\eta,x}(T)  &= \phi(X_T^{\tau,\eta,x},X^{\tau,\eta,x}(T)) 
\end{cases}
\, ,
\end{align}

where we have set for short
\[
\tilde{U}^{\tau,\eta,x}(t):=\int_{\RR_0} U^{\tau,\eta,x}(t,z)\delta(z) \nu(dz)\, .
\]

In particular the solution to the forward--backward SFDDE \eqref{FBsyst_tx} is the quadruple $(X,Y,Z,U)$ taking values in $\DDD\times \RR \times \RR \times \RR.$ 
We refer to \cite{delong} for a detailed introduction to forward-backward system with jumps.

Let us assume the following assumptions to hold:
\begin{ass}\label{assFB} \end{ass}

\begin{itemize}
\item[(B1)] The map $\psi: \, [0,T] \times \DDD \times \real \times \real \times \real \, \to \, \real$ is continuous 
and there exists $K>0$ and $m \geq 0$ such that

\begin{align*}
|\psi(t, \EXU, y_1, z_1, u_1) & - \psi(t, \EXD, y_2, z_2, u_2)| \leq K |(\EXU) - (\EXD)|_2 \\ 
 & + K(|y_1 -y_2| + |z_1 -z_2| + |u_1 - u_2|) \;;\\
|\psi(t, \EXU, y, z, u) & - \psi(t, \EXD, y, z, u)|\\ 
 & \leq K (1+|(\EXU)|_2 + |(\EXD)|_2 + |y|)^m \\ 
 & \cdot (1+|z|+|u|)(|(\EXU) - (\EXD)|_2)\;;\\
|\psi(t, 0,0, 0, 0,0)| & \leq K\, , 
\end{align*}

for all $(t, \eta_1,x_1, y_1, z_1, u_1), \, (t, \eta_2,x_2, y_2, z_2, u_2) \,\in\, [0,T] \times \DDD \times \real^3$;
\item[(B2)] the map	$\phi: \, \DDD \, \to \, \real$ is measurable and there exist $K>0$ and $m \geq 0$ such that 
\[
|\phi(\eta_1,x_1)-\phi(\eta_2,x_2)| \leq K(1+|(\eta_1,x_1)|_2 + |(\eta_2,x_2)|_2)^m |(\eta_1,x_1)-(\eta_2,x_2)|_2\, ,
\]
for all $( \eta_1,x_1), \, (\eta_2,x_2) \,\in\, \DDD $;

\item[(B3)] there exists $K>0$ such that the function $\delta : \RR_0 \to \RR$ satisfies
\[
|\delta(z)|  \leq K|(1 \wedge |z|) \, , \quad \delta (z) \geq 0 \, , \quad z \in \RR_0\, .
\]
\end{itemize}

\begin{rem}
Following \cite{delong}, we have chosen this particular form for the generator $\psi$ of the backward component 
in equation \eqref{FBsyst_tx}, due to the fact that it results to be convenient in many concrete applications.
\end{rem}

\begin{rem}
We want to stress that assumptions \ref{assFB} imply that there exists a suitable constant $C>0$ such that
\begin{align*}
|\psi(t,\eta,x,y,z,u)| & \leq C (1+|(\eta,x)|_2^{m+1} +|y|+|z|+|u|) \, ,\\ 
|\phi(\eta,x)| & \leq C (1+ |(\eta,x)|_2^{m+1})\, .
\end{align*}
\end{rem}

In what follows we will denote by $\KK([0,T])$ the space of all triplet $(Y,Z,U)$ of predictable stochastic processes taking value in 
$\RR \times \real \times \real$ and such that
\begin{align}
\nonumber
\norma (Y,Z,U) \norma^2_{\KK} & := \media \left[\sup_{t \,\in\, [0,T]} |Y(t)|^2\right] + \media\left[\int_0^T 
|Z(t)|^2 \de \tau \right] \\ \label{normK}
& + \media \left[\int_0^T \int_{\real_0} |U(t,z)|^2 \nu(dz)dt \right] < \infty \;.
\end{align}
%
%
The following Proposition ensures the existence and the uniqueness of the solution to the system \eqref{FBsyst_tx}, under suitable properties of the coefficients.
\begin{prop}\label{thm4.2}
Let us consider the coupled forward-backward system \eqref{FBsyst_tx} which satisfies Assumptions \ref{ass1} 
and Assumptions \ref{assFB}. 

Then, the coupled forward-backward system admits a unique solution 
\[
(X^{\tau,\EXZ }, Y^{\tau,\EXZ }, Z^{\tau,\EXZ}, U^{\tau,\EXZ}) \,\in\, \Sp\times\KK([0,T])\, .
\] 

Eventually we have that the map
\[
(\tau,\eta,x) \mapsto (X^{\tau,\eta,x},Y^{\tau,\EXZ }, Z^{\tau,\EXZ }, U^{\tau,\EXZ })\, ,
\]
is continuous.
\end{prop}
\begin{proof}
The existence and uniqueness of the solution to the forward component follows from theorem \ref{eu_f}, 
since Assumptions \ref{ass1} hold true by hypothesis, whereas the existence and uniqueness of the backward component under Assumptions \ref{assFB} follows \cite[Cor. 2.3]{Bar} or \cite[Thm. 4.1.3]{delong} . 

The continuity of the map $(\tau,\eta,x) \mapsto X^{\tau,\eta,x}$ is guaranteed by theorem \ref{eu_f}, whereas the continuity of $(\tau,\eta,x) \mapsto (Y^{\tau,\EXZ }, Z^{\tau,\EXZ }, U^{(\tau,\EXZ })$ follows from \cite[Prop. 1.1]{Bar}.
%

\end{proof}

\begin{thm}\label{thm}
Let us consider the coupled forward-backward system \eqref{FBsyst_tx} which satisfies Assumptions \ref{ass1} 
and \ref{assFB}. Let us define the function $u: \, [0,T]\times \DDD \, \to \, \real$, 
\[
u(t,\eta,x) 
:= Y^{t,\EXZ}_t \, ,
\]
with $t \,\in\, [0,T]$ and $(\EXZ) \,\in\, \DDD,$. 

Then, there exist $C >0$ and 
$m \geq 0$, such that, for any $t \in [0,T]$ and any $(\eta_1,x_1)$, $(\eta_2,x_2) \in \DDD$, the function $u$ satisfies

\begin{equation}\label{EQN:KFTTh2}
\begin{split}
&|u(t,\eta_1,x_1) - u(t,\eta_2,x_2)|\leq C|(\eta_1,x_1)-(\eta_2,x_2)|_2(1+ |(\eta_1,x_1)|_2 + |(\eta_2,x_2)|_2 )^m\, ,\\
&|u(t,0,0)| \leq C \, .
\end{split}
\end{equation}

Moreover, for every $t \, \in \, [0,T]$ and $(\EXZ) \, \in \, \DDD$ we have $\mathbb{P}-$a.s. and for a.e. $t \in [\tau,T]$
\begin{equation}\label{EQN:RepYXZ}
\begin{split}
Y^{\tau,\EXZ}(t) &= u\left (t, X^{\tau,\EXZ}_{t},\XTE(t)\right ) \, , \\
Z^{\tau,\EXZ}(t) &= \partial_x^\sigma u\left (t, \XTE_{t},\XTE(t)\right ) \, , \\ 
U^{\tau,\EXZ}(t,z) &= u(t,\XTE_{t},\XTE(t) + \gamma(t,\XTE_{t},\XTE(t),z) ) \\
& - u(t, \XTE_{t},\XTE(t))\, ,
\end{split}
\end{equation}
where $\partial_x^\sigma$ is the \textit{generalized directional gradient} in the sense of equation \eqref{EQN:GenGra}.
\end{thm}
\begin{rem}
Let us recall that, if $u$ is sufficiently regular, then 
\[
Z^{(\tau,\EXZ)}(t) = \partial_x u(t, \XTE_{t},\XTE(t)) \sigma(t, \XTE_{t},\XTE(t))\, .
\]
\end{rem}

\begin{proof}
The fact that $u(t,\eta,x) := Y^{t,\EXZ}_t$ satisfies \eqref{EQN:KFTTh2} immediately follows from the continuity of the map
\[
(\tau,\eta,x) \mapsto (X^{\tau,\eta,x},Y^{\tau,\EXZ }, Z^{\tau,\EXZ }, U^{\tau,\EXZ })\, ,
\]
proved in proposition \ref{thm4.2} together with assumptions \ref{ass1}.

The representation of $Y$ and $Z$ follow from \cite[Cor. 4.3]{FMT}. 

As regards the process $U,$ using 
the standard notion of joint variation we have
\begin{equation}\label{EQN:1b}
\left\langle Y^{\tau,\EXZ} (\cdot),J(\cdot) \right\rangle_{[\tau,T]} = \int_\tau^T \int_{\RR_0} z \,U^{\tau,\EXZ}(s,z) N(ds,dz)\, .
\end{equation}

On the other hand, Theorem \ref{thm3.1} implies 

\begin{align}
\nonumber
& \left\langle u(\cdot, X^{\tau,\EXZ}_{\cdot}, \XTE(\cdot)),J(\cdot) \right\rangle_{[\tau,T]} \\ \nonumber
& = \int_\tau^T \int_{\RR_0} z \left [ u(s,\XTE_s, \XTE(s) + \gamma(s,\XTE_s,\XTE(s),z) )\right] N(ds,dz) \\ 
\label{EQN:2b}
& - \int_\tau^T \int_{\RR_0} z \left[ u(s, \XTE_s,\XTE(s))\right ] N(ds,dz)\, . 
\end{align}

Comparing now equation \eqref{EQN:1b} and equation \eqref{EQN:2b}, the representation for $U$ in equation \eqref{EQN:RepYXZ} follows.
\end{proof} 

\subsection{The non-linear Kolmogorov equation} \label{nonlin}
The present section is devoted to prove that the solution to the \fb system \eqref{FBsyst_tx} 
can be connected to the solution of a path-dependent partial integro-differential equation with values in the Hilbert space $\DDD$.

More precisely, let us consider the Markov process $(\XTE_t,\XTE(t))$ defined as the solution of equation 
\eqref{EQN:DelayGenerale}, and the corresponding infinitesimal generator $\LL_t$.

The path-dependent partial-integro differential equation we want to investigate has the following form

\begin{equation}\label{parabolic2}
\begin{cases}
\frac{\partial}{\partial t} u(t,\eta,x) + \LL_t u(t,\eta,x) = \psi\left (t,\eta,x, u(t,\eta,x), \partial_x^\sigma u(t,\eta,x), \mathcal{J} u(t, \eta,x)\right ) \, , \\ 
u(T,\eta,x) = \phi(\eta,x), 
\end{cases}
\end{equation}
for all $t\, \in \, [0,T],$ and $(\eta,x) \, \in \, \DDD,$ where $u : [0,T]\times \DDD \,\to, \real$ is an 
unknown function, $\psi$ and $\phi$ are two given functions such that $\psi: [0,T] \times \DDD \times \real \times \real \times \real \,\to \real$ and 
$\psi : \DDD \,\to \real$, $\partial_x^\sigma u$ is the \textit{generalized directional gradient} and $\mathcal{J}$ is a functional acting as
\[
\mathcal{J} u(t,\eta, x) = \int_{\RR_0} \left (u(t, \eta, x+ \gamma(t,\eta,x,z)) - u(t,\eta, x)\,\right ) \delta(z)\nu(dz)\, .
\]
%

In particular, we want to look for a {\it mild solution} of equation (\ref{parabolic2}), according to the 
following definition.
\begin{defn}\label{mild}
A \emph{mild solution} 
to equation \eqref{parabolic2} is a function $u:\, [0,T] \times \DDD \,\to \real$ such that there exist $C >0$ 
and $m \geq 0$, such that, for any $t \in [0,T]$ and any $(\eta_1,x_1)$, $(\eta_2,x_2) \in \DDD$, $u$ 
satisfies

\begin{equation}
\begin{split}
&|u(t,\eta_1,x_1) - u(t,\eta_2,x_2)|\leq C|(\eta_1,x_1)-(\eta_2,x_2)|_2(1+ |(\eta_1,x_1)|_2 + |(\eta_2,x_2)|_2 )^m \\
&|u(t,0,0)| \leq C 
\end{split}
\end{equation}

and the following identity hold true
\begin{equation}\label{def_mild}
u(t,\eta,x) = P_{t,T} \phi (\eta,x) + \int_t^T P_{t,s}[\psi(\cdot, u(s,\cdot), \partial_x^\sigma u(s,\cdot), \mathcal{J} u(s, \cdot )](\eta,x) \de s \;,
\end{equation}
for all $t\, \in \, [0,T],$ and $(\eta,x) \, \in \, \DDD$ and where $P_{t,s}$ is the Markov semigroup for equation \eqref{delay_eq} introduced in equation \eqref{semigroup}. 
\end{defn}

\begin{thm}\label{mild_ex}
Assume that Assumptions \ref{ass1} and Assumptions \ref{assFB} hold true. Then, the path-dependent 
partial integro-differential equation \eqref{parabolic2} admits a unique mild solution $u,$ in the 
sense of definition \ref{mild}. In particular, the mild solution $ u$ coincide with the function $u$ introduced in theorem \ref{thm}.
\end{thm}

\begin{proof} 
In what follows, as above, we will denote for short
\[
\tilde{U}^{\tau,\eta,x}(s) :=\int_{\RR_0}U^{\tau,\eta,x}(s,z)\delta(z) \nu(dz)\, ,
\]
Let us consider the backward stochastic differential equation in equation (\ref{FBsyst_tx}), namely,  

\[
\begin{split}
Y^{t,\EXZ} (t) &= \phi(X^{t,\EXZ}_T, X^{t,\eta,x}(T))  + \\
&+ \int_t^T \psi\left(X^{t,\eta,x}_s,X^{t,\eta,x}(s), Y^{t,\EXZ}(s), Z^{t,\EXZ}(s),\tilde{U}^{t,\eta,x}(s) \right) \de s \\ 
& + \int_t^T Z^{t,\EXZ}(s) \de W(s) + \int_t^T \int_{\real_0} U^{t,\EXZ}(s,z) \comp(\de s, \de z) \;.
\end{split}
\]

Taking the expectation and exploiting equation \eqref{EQN:RepYXZ}, then $Y$ satisfies equation \eqref{def_mild}.

In order to show the uniqueness let $u(t,\eta,x), \; 0 \leq \tau \leq t \leq T,$ be a mild solution 
of equation \eqref{parabolic2}, so that 

\begin{align*}
u(t,&\eta,x) = \media\left[  \phi(X^{t,\EXZ}_T, X^{t,\eta,x}(T))\right] \\
+ & \media \left[ \int_t^T \psi\left(X^{t,\eta,x}_s,X^{t,\eta,x}(s), Y^{t,\EXZ}(s), Z^{t,\EXZ}(s),\tilde{U}^{t,\eta,x}(s) \right)\de s \right] \;.
\end{align*}

By recalling that $\left (X^{\tau,\eta,x}_t,X^{\tau,\eta,x}(t)\right )_{t \in [0,T]}$ is a $\DDD-$Markov process, and denoting by $\mathbb{E}^t$ the 
conditional expectation w.r.t. the filtration $\mathcal{F}_t$, we can write

\[
\begin{split}
u(t & ,X^{t,\EXZ}_t,X^{t,\EXZ}(t))  =  \media^{t} \left[ \phi(X^{t,\EXZ}_T, X^{t,\eta,x}(T)) \right ] \\
+ & \media^{t} \left[ \int_\tau^T \psi\left(X^{t,\eta,x}_s,X^{t,\eta,x}(s), Y^{t,\EXZ}(s), Z^{t,\EXZ}(s),\tilde{U}^{t,\eta,x}(s) \right)ds \right] \\ 
- & \media^{t}\left[ \int_\tau^t \psi\left(\XTE_s,\XTE(s), Y^{\tau,\EXZ}(s), 
Z^{\tau,\EXZ}(s),\tilde{U}^{\tau,\eta,x}(s) \right) \de s \right] \; .
\end{split}
\]

We set, for short,
\begin{align*}
\xi & := \phi(X^{t,\EXZ}_T, X^{t,\eta,x}(T)) \\ 
+ & \int_\tau^T  \psi\left(X^{t,\eta,x}_s,X^{t,\eta,x}(s), Y^{t,\EXZ}(s), Z^{t,\EXZ}(s),\tilde{U}^{t,\eta,x}(s) \right) \de s \, .
\end{align*}

Thanks to the martingales representation theorem, see, e.g., \cite[Thm. 5.3.5]{App}, there exist two 
predictable processes $\bar{Z} 
\,\in\, L^2(\Omega\times[0,T])$ and $ \, \bar{U} \,\in\, L^2(\Omega\times[0,T]\times\real_0)$ such that 

\begin{align*}
u(t,& \XTE_t,\XTE(t)) = u(\tau,\eta,x) \\ 
 & + \int_\tau^t \bar{Z}^{\tau,\EXZ}(s) \de W(s) + \int_\tau^t \int_{\real_0} \bar{U}^{\tau,\EXZ}(s, z) 
\comp(\de s, \de z) \\ 
 & - \int_\tau^t \psi\left(\XTE_s,\XTE(s), Y^{\tau,\EXZ}(s), Z^{\tau,\EXZ}(s),\tilde{U}^{\tau,\eta,x}(s) \right) \de s \;.
\end{align*}

Applying theorem \ref{thm3.1}, we have 

\begin{align*}
&u(t,\XTE_t,\XTE(t)) = \phi (\XTE_T,\XTE(T))+\\
& - \int_t^{T} \partial_x^\sigma u(s,\XTE_s,\XTE(s)) \de W(s) \\
& - \int_t^T \int_{\real_0} \left[u(s,\XTE_s,\XTE(s) +\gamma(s,\XTE_s,\XTE(s),z)) \right. \\ 
& - \left. u(s, \XTE_s,\XTE(s)) \right] \tilde{ N}(\de s, \de z) \\ 
& +  \int_t^T \psi\left(\XTE_s,\XTE(s), Y^{\tau,\EXZ}(s), Z^{\tau,\EXZ}(s),\tilde{U}^{\tau,\eta,x}(s) \right) \de s \;.
\end{align*}

By comparing last equation with the backward component of equation \eqref{FBsyst_tx}, we note that 
$(Y^{\tau,\EXZ}(t),Z^{\tau,\EXZ}(t),U^{\tau,\EXZ}(t,z))$ and the following three functions 
\[
\begin{split}
&\left (u(t,\XTE_t,\XTE(t)), \partial_x^\sigma u(t,\XTE_t,\XTE(t))\right. , \\
&u\left .(t,\XTE_s,\XTE(s)+\gamma(s,\XTE_s,\XTE(s),z)) - u(t, \XTE_s,\XTE(s))\right )\, ,
\end{split}
\]
solve the same equation. Therefore, due to the uniqueness of the solution, we have that
\[
Y^{\tau,\EXZ}(t) = u(t, \XTE_t,\XTE(t)) \;.
\]
Setting $\tau = t,$ we obtain $Y^{\tau,\EXZ}(t) = u(t,\eta,x)$ and the proof is complete.
\end{proof}

\section{Application to optimal control}\label{SEC:OC} 

We are to apply previously derived results to a general class on non-linear control problem. The present section closely follows in \cite[Section. 7]{FT1}, in particular we will consider weak control problems, we refer to \cite{Fle} for a general treatment of the present notion of control, or \cite{CDPN1,CDPN2,KLSu,KLSu2}.

Let us therefore consider the following $\RR-$valued controlled delay equation, 
\begin{equation}\label{EQN:ControlDelay}
\begin{cases}
d X(t) &= \left (\mu(t,X_t,X(t)) + F(t,X(t),\alpha(t))\right )dt +\\
&\quad+ \sigma(t,X(t))dW(t) + \int_{\RR_0} \gamma(t,X(t),z) \tilde{N}(dt,dz)\, ,\\
\left (X_{t_0},X({t_0})\right ) &= \left (x,\eta\right )\,, 
\end{cases}
\end{equation}
where we have denoted by $\alpha : \Omega \, \times \, [0,T] \to \mathcal{A}$ a $\left (\mathcal{F}_t\right )_{t \geq 0}-$predictable process representing the control, being $\mathcal{A} \subset \RR^N$ a convex set, $N \in \mathbb{N}$. 

In what follows we assume $\mu$, $\sigma$ and $\gamma$ to satisfy assumptions \ref{ass1}, we also require that it exists a constant $C_\sigma >0$ such that, for any $t \in [0,T]$ and $x \in \RR$,
\[
|\sigma^{-1}(t,x)|\leq C_\sigma\, .
\]

We remark that a possibly choice for the coefficient $\mu$ in equation \eqref{EQN:ControlDelay} is of the form
\[
\mu(t,X_t,X(t)) = \int_{-r}^0 X(t+\theta)\varpi(d\theta) \,,
\]
for $\varpi$ a Borel measure of bounded variation on the interval $[-r,0]$.

Following \cite[Section 7]{FT1}, we will say therefore that an \textit{admissible control system} (acs) is given by $\mathbb{U} = \left (\Omega,\mathcal{F},\mathbb{P},W,\nu,\alpha,X\right )$, where $\left (\Omega,\mathcal{F},\mathbb{P}\right )$ is a complete probability space, with an associated filtration satisfying usual conditions, $W$ is a Wiener process whereas $\nu$ is a L\'{e}vy measure also satisfying usual assumptions introduced in previous sections, $\alpha$ is the control defined above and $X$ is the unique solution to equation \eqref{EQN:ControlDelay}. Then we wish to minimize, over all control $\alpha \in \mathcal{A}$, the following functional
\begin{equation}\label{EQN:FinMin}
J\left (t_0,\left (x,\eta\right ),\mathbb{U}\right ) = \int_0^T h(s,X(s),\alpha) ds + g(X(T))\, .
\end{equation}

We thus assume the following to hold.

\begin{ass}
\begin{description}
\item[(i)] let $F:[0,T] \times \RR \times \mathcal{A} \to \RR$ be measurable and such that there exist $C_F>0$ and $m \geq 0$ such that, for any $t \in [0,T]$, $x,\,x_1,\,x_2 \in \RR$ and $\alpha \in \mathcal{A}$,
\[
\begin{split}
|F(t,x,\alpha)| &\leq C\, ,\\
|F(t,x_1,\alpha) - F(t,x_2,\alpha)| &\leq C_F(1 + |x_1| + |x_2|)^m |x_1-x_2|\,.
\end{split}
\]

\item[(ii)] let $h:[0,T] \times \RR \times \mathcal{A} \to \RR \cup \{+\infty\}$ be measurable and such that there exist $C_h >0$ and $m \geq 0$ such that, for any $t \in [0,T]$, $x,\,x_1,\,x_2 \in \RR$ and $\alpha \in \mathcal{A}$,
\[
\begin{split}
h(t,0,\alpha) &\geq -C_h\, ,\quad \inf_{\alpha \in \mathcal{A}} h(t,0,\alpha) \leq C_h\, ,\\
|h(t,x_1,\alpha) - h(t,x_2,\alpha)| &\leq C_h(1 + |x_1| + |x_2|)^m |x_1-x_2| + h(t,x_2,\alpha)\,.
\end{split}
\]

\item[(iii)] let $g:\RR \to \RR$ be measurable and such that there exist $C_g>0$ and $m>0$ such that, for any $x,\,x_1,\,x_2 \in \RR$  it holds
\[
|g(x_1) -g(x_2) | \leq C_g (1 + |x_1| + |x_2|)^m |x_1-x_2|\, .
\]
\end{description}
\end{ass}

The particular form for equation \eqref{EQN:ControlDelay} leads to consider an associated \textit{Hamilton-Jacobi-Bellman} (HJB) equation which is a semilinear partial integro-differential equation of the form of equation \eqref{parabolic2} studied in previous sections. Noticed that the particular form for equation \eqref{EQN:ControlDelay}, in particular the presence of the control in the drift, is imposed by the techniques we will use.

Then the controlled equation \eqref{EQN:ControlDelay}, together with the functional $J$ introduced in equation \eqref{EQN:FinMin}, lead to define in a classical way the \textit{Hamiltonian} associated to the above problem as
\[
\begin{split}
\psi(t,x,z) &= -\inf_{\alpha\in \mathcal{A}} \left \{h(t,x,\alpha) + z \sigma^{-1}(s,x) F(s,x,\alpha) \right \}\, ,\\
\Gamma(t,x,z) &= \left \{ \alpha \in \mathcal{A} \, : \, \psi(t,x,z) + h(t,x,\alpha) + z \sigma^{-1}(s,x) F(s,x,\alpha)=0 \right \}\, .
\end{split}
\]

Let us stress that under above assumptions we have that $\psi$ satisfies assumptions \ref{assFB}. Eventually we can formulate the HJB equation associated to the above stated non-linear control problem to be
\begin{equation}\label{EQN:MildHJM}
\begin{cases}
\frac{\partial}{\partial t} u(t,\eta,x) + \mathcal{L}_t u(t,\eta,x)&=\psi(t,x,\partial_x^{\sigma^{-1} F} u(t,\eta,x))\, ,\\
u(T,\eta,x) = g(x)\, ,
\end{cases}
\end{equation}
where the notation is as above introduced. From Theorem \ref{mild_ex}, it follows that equation \eqref{EQN:MildHJM} admits a unique solution in the sense of generalized direction gradient.

Eventually, from \cite[Theorem 7.2]{FT} or \cite[Theorem 7.2]{FT1} which follow in a straightforward manner in the present case, we have that an acs system is optimal if and only if
\[
\alpha(t) \in \Gamma\left (t,X(t),\zeta(t,X_t,X(t))\right )\,,
\]
being $\zeta : [0,T] \times \DDD \to \RR$ an element of the directional generalized gradient.

\begin{thm}\label{THM:Synth}
Let $u$ be a mild solution to the HJB equation \eqref{EQN:MildHJM}, and choose $\zeta$ to be an element of the generalized directional gradient $\partial_x^{\sigma^{-1}F} u$. Then, for all acs, we have that $J(t_0,x,\eta,\mathbb{U}) \geq u(t_0,x,\eta)$, and the equality holds if and only if
\[
\alpha(t) \in \Gamma\left (t,X(t),\zeta(t,X_t,X(t))\right )\, ,\quad \mathbb{P}-\,a.s. \, \mbox{ for a.a. } \, t \in [t_0,T]\, .
\]

Moreover, if there exists a measurable function $\varsigma \,: [0,T] \times \RR \to \mathcal{A}$ with
\[
\varsigma(t,x,z) \in \Gamma(t,x,z)\, ,
\]
then there also exists at least one acs such that
\[
\bar \alpha(t) = \varsigma(t,X^\alpha(t),\zeta(t,X^\alpha_t,X^\alpha(t)))\, ,\quad \mathbb{P}-\,a.s. \, \mbox{ for a.a. } \, t \in [t_0,T]\, ,
\]
where $\left (X^\alpha_t,X^\alpha(t)\right )$ is the solution to equation 
\[
\begin{cases}
d X^\alpha(t) &= \mu(t,X^\alpha_t,X^\alpha(t))dt + \\
&+ F(t,X^\alpha(t),\varsigma(t,X^\alpha(t),\zeta(t,X^\alpha(t))))dt +\\
&+ \sigma(t,X^\alpha(t))dW(t) + \int_{\RR_0} \gamma(t,X^\alpha(t),z) \tilde{N}(dt,dz)\, ,\\
\left (X^\alpha_{t_0},X^\alpha({t_0})\right ) &= \left (x,\eta\right )\,, 
\end{cases}
\]
\end{thm}
\begin{proof}
See \cite[Th. 7.2]{FT1} or also \cite[Th. 4.7, Cor. 4.8]{KLSu}.
\end{proof}

%


\end{document}